\documentclass[12pt]{amsart}
\usepackage[english]{babel}
\usepackage{amsmath}
\usepackage{amsthm}
\usepackage{anysize}
\usepackage{amsfonts}
\usepackage{amssymb}
\usepackage{MnSymbol}
\usepackage{color}
\usepackage{float}
\usepackage[T1]{fontenc}
\usepackage{graphicx}
\usepackage{hyperref}
\usepackage{xr}
\usepackage{tikz} 
\usetikzlibrary{matrix}

\numberwithin{equation}{section}
\newtheorem{theorem}{\textit{Theorem}}[section]
\newtheorem{proposition}[theorem]{\textit{Proposition}}

\newtheorem{lemma}[theorem]{\textit{Lemma}}
\newtheorem{corollary}[theorem]{\textit{Corollary}}

\newtheorem{remark}[theorem]{\textit{Remark}}

\marginsize{3.3cm}{3.3cm}{3cm}{3cm}

\title{On the topology of some hyperspaces of convex bodies associated to tensor norms}
\author{Luisa F. Higueras-Monta{\~n}o and Natalia Jonard-P{\'e}rez}

\address{Departamento de  Matem\'aticas,
Facultad de Ciencias, Universidad Nacional Aut\'onoma de M\'exico, 04510 Ciudad de M\'exico, M\'exico.}

\email{(L. Higueras-Monta\~no) fher@cimat.mx}
\email{(N.\,Jonard-P\'erez) nat@ciencias.unam.mx}

\thanks{
The first author has been supported by The Post-Doctoral Scholarship Program at UNAM. The first and second author have been supported by CONACyT grant 252849 (M\'exico) and by  PAPIIT grant IN115819  (UNAM, M\'exico)}



\begin{document}

\begin{abstract}
For every tuple $d_1,\dots, d_l\geq 2,$ let $\mathbb{R}^{d_1}\otimes\cdots\otimes\mathbb{R}^{d_l}$ denote the tensor product of $\mathbb{R}^{d_i},$ $i=1,\dots,l.$ Let us denote by $\mathcal{B}(d)$ the hyperspace of centrally symmetric convex bodies in  $\mathbb{R}^d,$ $d=d_1\cdots d_l,$ endowed with the Hausdorff distance, and by
$\mathcal{B}_\otimes(d_1,\dots,d_l)$ the subset of $\mathcal{B}(d)$ consisting of the convex bodies that are closed unit balls of reasonable crossnorms on $\mathbb{R}^{d_1}\otimes\cdots\otimes\mathbb{R}^{d_l}.$ It is known that $\mathcal{B}_\otimes(d_1,\dots,d_l)$ is a closed, contractible and locally compact subset of $\mathcal{B}(d).$ The hyperspace $\mathcal{B}_\otimes(d_1,\dots,d_l)$ is called \textit{the space of tensorial bodies}. In this work we determine the homeomorphism type of $\mathcal{B}_\otimes(d_1,\dots,d_l).$ 
We show that even if $\mathcal{B}_\otimes(d_1,\dots,d_l)$  is not closed with respect to the Minkowski sum, it is  an absolute retract homeomorphic to $\mathcal{Q}\times\mathbb{R}^p,$ where $\mathcal{Q}$ is the Hilbert cube and $p=\frac{d_1(d_1+1)+\cdots+d_l(d_l+1)}{2}.$ 
Among other results, the relation between the Banach-Mazur compactum and the Banach-Mazur type compactum associated to $\mathcal{B}_\otimes(d_1,\dots,d_l)$ is examined. 
\end{abstract}

\maketitle














\section{Introduction}
\label{sec:introduction}

In the last 50 years a lot of research on the topological properties 
of some hyperspaces of convex compacta in the Euclidean space $\mathbb{R}^d$ has been carried out.
Among others,  the topological structure of the following hyperspaces has been determined: the non-empty compact convex sets $cc(\mathbb{R}^d)$ \cite{nadler1979}; the centrally symmetric convex bodies $\mathcal{B}(d)$ \cite{banacamazurcompactum,antonyananerviowest};  the class of convex bodies $\mathcal{K}_d$ \cite{AntonyanNatalia,Macbeath51}; the convex bodies of constant width \cite{antonyanjonardordonez,Bazilevich97,Bazylevych2006}; and, recently, the convex bodies with smooth boundary \cite{Belegradek2018}.
As can be traced in the seminal work \cite{nadler1979}, the fundamental tool to tackle these problems has been the theory of Hilbert cube manifolds. A Hilbert cube manifold is a separable and metrizable space  modeled on the Hilbert cube (see Section \ref{sec:notation}). The Hilbert cube $\mathcal{Q}$ is the space $\prod_{i=1}^{\infty}[-1,1]$ endowed with the product topology. 

\

 Some of the previous hyperspaces arise naturally in contexts different from Convex Geometry or Geometric Topology. This is the case of $\mathcal{B}(d),$ which emerges from the theory of Banach spaces. It  is well-known that the convex bodies in $\mathcal{B}(d)$ are, precisely, the unit balls of norms on $\mathbb{R}^d$ (see for instance \cite[Section 1.7]{Schneider1993} or \cite[Section 1.1]{ThompsonAC}). In the same vein, a hyperspace arising from the theory of tensor norms on Banach spaces was introduced in \cite{tensorialbodies,topologytensorialbodies}. It is called \textit{the space of tensorial bodies} 
 $\mathcal{B}_\otimes(d_1,\dots,d_l),$ $d_i\in\mathbb{N},$ $i=1,\dots,l,$ and it is composed by the convex bodies $P\subset\mathbb{R}^d,$ $d=d_1\cdots d_l,$ that are unit balls of Banach spaces, whose norms are reasonable crossnorms.
The main goal of this paper is to calculate the topological structure of this hyperspace. Our approach to the problem is based on the theory of Hilbert cube manifolds and topological groups (see Sections \ref{sec:notation} and \ref{sec:Gspaces}). 

\

Before giving a formal definition of $\mathcal{B}_\otimes(d_1,\dots,d_l)$, it is convenient to recall some facts about the so called theory of tensor norms. We are interested in the finite-dimensional setting.
A norm $\alpha(\cdot)$ on the tensor product of finite-dimensional Banach spaces $\otimes_{i=1}^{l}(\mathbb{R}^{d_i},\|\cdot\|_i)$ is called a \textit{reasonable crossnorm} if and only if
\begin{equation*}
\label{eq:Caracterizacion de normars razonables cruzadas}
\epsilon\left(u\right)\leq\alpha\left(u\right)\leq\pi\left(u\right)\text{ for every }u\in\otimes_{i=1}^{l}(\mathbb{R}^{d_i},\|\cdot\|_i).
\end{equation*}
The norms $\epsilon(\cdot)$ and $\pi(\cdot)$ are the injective and the projective tensor norms, respectively. Actually, they are the smallest and the biggest reasonable crossnorms (we refer to \cite{defantfloret,Ryan2013} for the definitions and basic properties). In \cite[Theorem 3.2]{tensorialbodies}, a characterization of the convex bodies in $\otimes_{i=1}^{l}\mathbb{R}^{d_i}\simeq\mathbb{R}^{d},$ $d=d_1,\dots,d_l,$ that are closed unit balls of reasonable crossnorms is given. That is, an explicit description of the convex bodies $P\subset\otimes_{i=1}^{l}\mathbb{R}^{d_i}$ for which there exists norms $\|\cdot\|_i,$ $i=1,\dots,l,$ (not determined a priori)  such that $P$ is the closed unit ball of a reasonable crossnorm on $\otimes_{i=1}^{l}(\mathbb{R}^{d_i},\|\cdot\|_i)$ is established. With it, the class of tensorial bodies in $\otimes_{i=1}^{l}\mathbb{R}^{d_i}$ is introduced \cite[Definition 3.3]{tensorialbodies}. A centrally symmetric ($0$-symmetric) convex body  $P\subset\otimes_{i=1}^{l}\mathbb{R}^{d_i}$ is called \textit{tensorial body} if there exists $0$-symmetric convex bodies $P_i\subset\mathbb{R}^{d_i},$ $i=1,\dots,l,$ such that 
$$
P_{1}\otimes_{\pi}\cdots\otimes_{\pi}P_{l}\subseteq P\subseteq P_{1}\otimes_{\epsilon}\cdots\otimes_{\epsilon}P_{l}.
$$
Where, $\otimes_{\pi}$ and $\otimes_{\epsilon}$ are the projective and the injective tensor products of convex bodies, respectively (see Section \ref{sec:basictensorialbodies}). The hyperspace associated to the tensorial bodies $\mathcal{B}_\otimes(d_1,\dots.,d_l)$ is introduced in \cite{topologytensorialbodies}.  There, some of its fundamental properties are exhibited. It is known that $\mathcal{B}_\otimes(d_1,\dots,d_l)$ carries a natural structure of a $GL_\otimes$-proper space. In this case, $GL_\otimes(d_1,\dots,d_l)$ is the group consisting of linear isomorphisms on $\otimes_{i=1}^{l}\mathbb{R}^{d_i}$ preserving decomposable vectors (see Section \ref{sec:Gspaces}). The action is given by the evaluation $TP,$ for every $T\in GL_\otimes(d_1,\dots,d_l)$ and $P\in\mathcal{B}_\otimes(d_1,\dots,d_l).$ The associated orbit space $\mathcal{B}_\otimes(d_1,\dots,d_l)/GL_\otimes(d_1,\dots,d_l)$ is a metrizable compact space called \textit{the compactum of tensorial bodies} \cite[Corollary 4.9]{topologytensorialbodies}.

The main results of the paper are Theorem \ref{thm:ANRproperty} and Corollary \ref{cor:mainresultcor}. In Theorem \ref{thm:ANRproperty}, we show that $\mathcal{B}_\otimes(d_1,\dots,d_l)$  is an absolute retract (see Section \ref{sec:notation} for the definition). In Corollary \ref{cor:mainresultcor}, we calculate the topological structure of the space of tensorial bodies. It is proved that 
\begin{equation}
\label{eq:pvalue}
\mathcal{B}_\otimes(d_1,\ldots,d_l)\textit{ is homeomorphic to }\mathcal{Q}\times\mathbb{R}^p,\ 
 p=\frac{d_1(d_1+1)+\cdots+d_l(d_l+1)}{2},
\end{equation}
for every tuple  $d_i\geq2,$ $i=1,\dots,l.$ It is known that  in $\mathbb{R}\otimes\mathbb{R}^d$  the tensor structure is trivial, and $\mathcal{B}_\otimes(1,d)$ coincides with $\mathcal{B}(d)$ \cite[Proposition 3.8]{tensorialbodies}.

We now briefly describe the contents of the paper. In Sections \ref{sec:notation} and \ref{sec:Gspaces}, we introduce the notation and basic results on Hilbert cube manifolds and the theory of $G$-spaces. Then, in Section \ref{sec:basictensorialbodies}, we review the fundamental properties of the space of tensorial bodies \cite{tensorialbodies,topologytensorialbodies}. In Proposition \ref{prop:resumeprop} and Theorem \ref{thm:resumenthm}, we summarize the results from \cite{topologytensorialbodies} that we use frequently in the work. In this sense, it is recalled that
\begin{equation}
\label{eq:introLtensor}
\mathcal{B}_\otimes(d_1,\dots,d_l)\textit{ is homeomorphic to }\mathcal{L}_\otimes(d_1,\dots,d_l)\times\mathbb{R}^p
\end{equation}
 where $\mathcal{L}_\otimes(d_1,\dots,d_l)$ is the $O_\otimes$-global slice for the $GL_\otimes$-space $\mathcal{B}_\otimes(d_1,\dots,d_l)$ introduced in \cite[Section 4.1]{topologytensorialbodies}, and $p$ is the same as (\ref{eq:pvalue}). Here, $O_\otimes(d_1,\dots,d_l)$ is the closed subgroup of $GL_\otimes(d_1,\dots,d_l)$ consisting of orthogonal maps (Section \ref{sec:Gspaces}). Through the work, the letters $d, d_i\geq2,$ $i=1,\dots,l,$ will denote natural numbers. The proposition \ref{prop:continuity of hilbert} is new. It shows that the Hilbertian tensor product $\otimes_2$ of $0$-symmetric ellipsoids is continuous with respect to the Hausdorff distance.

In Section \ref{sec:relatedhyperspaces}, we completely describe the topological structure of the hyperspace $\Pi(d_1,\dots,d_l)=\{P_{1}\otimes_{\pi}\cdots\otimes_{\pi}P_{l}:P_i\in\mathcal{B}(d_i)\},$ determined by the projective tensor product $\otimes_\pi$. We show that  the
$GL_\otimes$-map
$conv_\otimes:\mathcal{B}_\otimes(d_1,\dots,d_l)\rightarrow\Pi(d_1,\dots,d_l),$ defined in  (\ref{eq:tensorialconvexhull}), is a $GL_\otimes$-deformation retraction onto $\Pi(d_1,\dots,d_l),$ and it is also a proper map (Proposition \ref{prop:convisproper}). The main result of this section is Theorem \ref{thm:structurePi}. There, 
we prove that $\Pi(d_1,\dots,d_l)$ is a Hilbert cube manifold homeomorphic to $\mathcal{Q}\times\mathbb{R}^p,$ with $p$ as in (\ref{eq:pvalue}).

Determining whether or not a metrizable locally compact space is an absolute neighborhood retract ANR (see Section \ref{sec:notation} for the definition) is a fundamental step in order to prove that it is a Hilbert cube manifold. Section \ref{sec:ANRproperty} is devoted to the study of the  ANR property on $\mathcal{B}_\otimes(d_1,\dots,d_l)$. First, in Proposition \ref{prop:notconvex}, we show that  $\mathcal{B}_\otimes(d_1,\dots,d_l)$ is not closed under the Minkowski sum. Establishing, in consequence, a notable difference with other hyperspaces such as $\mathcal{K}_d$ and $\mathcal{B}(d).$  Despite this failure, in Proposition \ref{prop:multilineality}, we give conditions under which  the Minkowski sum of tensorial bodies remains in the class. Finally, in Theorem \ref{thm:ANRproperty}, we prove that  $\mathcal{B}_\otimes(d_1,\dots,d_l)$ is an absolute retract. This is done by means of exhibiting a retraction from the space of $0$-symmetric convex bodies in $\otimes_{i=1}^{l}\mathbb{R}^{d_i}$ onto the space of tensorial bodies. 

The topological structure of $\mathcal{B}_\otimes(d_1,\dots,d_l)$ is calculated in Section \ref{sec:mainresults}. Our approach consist in proving that $\mathcal{L}_\otimes(d_1,\dots,d_l)$ is homeomorphic to $\mathcal{Q}$  (see (\ref{eq:introLtensor}) above). To address this problem, we use Torunczyk's characterization of Hilbert cube manifolds \cite[Theorem 1]{Torunczyk1980} and a classical result of R. Edwards \cite[\textsection 3]{Torunczyk1980}. In this way, in Proposition \ref{prop:zetaset}, we first exhibit a $Z$-set in $\mathcal{L}_\otimes(d_1,\dots,d_l),$ then we prove that the complement of such a set is an open Hilbert cube manifold in $\mathcal{L}_\otimes(d_1,\dots,d_l).$ This allows us to show that
$$
\mathcal{L}_\otimes(d_1,\dots,d_l) \textit{ is homeomorphic to }\mathcal{Q}.
$$
This result appears as Theorem \ref{thm:maintheorem}. In Proposition \ref{prop:Ltensorpoyectivecontraible},
it is showed that the hyperspace $\Pi(d_1,\dots,d_l)\cap\mathcal{L}_\otimes(d_1,\dots,d_l)$ is an $O_\otimes$-deformation retraction of $\mathcal{L}_\otimes(d_1,\dots,d_l),$ and it is homeomorphic to $\mathcal{Q}.$

We finish the paper by exploring the topological structure of the compactum of tensorial bodies $\mathcal{B}_\otimes(d_1,\dots.,d_l)/GL_\otimes(d_1,\dots,d_l).$ Among others, we show that it is a contractible space (Proposition \ref{prop_BMtensorcontractible}). Some remarks  and open questions on the relation between the compactum  $\mathcal{B}_\otimes(d_1,\dots.,d_l)/GL_\otimes(d_1,\dots,d_l)$ and the so called Banach-Mazur compactum $\mathcal{BM}(d)$ are detailed. In Proposition \ref{prop:orbitspacePI}, we show that the orbit space
$\Pi(d_1,\dots,d_l)/GL_\otimes(d_1,\dots,d_l)$ is homeomorphic to $\mathcal{BM}(d_1)\times\cdots\times\mathcal{BM}(d_l),$ as long as the dimensions $d_i,$ $i=1,\dots,l,$ are different from each other.

\subsection{Notation and basic definitions}
\label{sec:notation}

Let $X$ be a topological space and let $A\subseteq X$ be closed. We say that $A$ is a retract of $X$ if there exists a continuous surjective map $r:X\rightarrow A$ such that $r(a)=a,$ for all $a\in A.$ Such a map $r$ is called a retraction. 
A metrizable space $X$ is called an \textit{absolute neighborhood retract} (ANR) if for every metrizable space $Z$ containing $X$ as closed subset, there exists a neighborhood $U$ of $X$ in $Z$ and a retraction $r:U\rightarrow X.$ In this case, we write $X\in ANR.$ If we can always take $U=Z,$ we say that $X$ is an \textit{absolute retract} (AR), and  we write $X\in AR.$ 
It is well-known that every closed convex subset of a Banach space is an AR \cite[Theorem 1.5.1]{vanMillbook}.

A continuous map $f:X\rightarrow Y$ between topological spaces is called \textit{proper} if for every compact subset $K\subseteq Y,$ $f^{-1}(K)$ is compact too. A proper and surjective map $f:X\rightarrow Y$  between ANR's is called \textit{cell-like} if every fiber $f^{-1}(y),$ $y\in Y,$ has the property $UV^{\infty}.$  That is, for each neighborhood $U$ of $f^{-1}(y)$ there exists a neighborhood $V\subset U$ of $f^{-1}(y)$ such that the inclusion $V\hookrightarrow U$ is homotopic to a constant map from $V$ to $U.$ In particular, if all of the fibers $f^{-1}(y)$ are contractible, then $f$ has the property $UV^{\infty}$ (see \cite[p. 91]{Chapmanbook}). Recall that a topological space is called \textit{contractible} if there exists a homotopy $H:X\times[0,1]\rightarrow X$ such that $H(\cdot,0)$ is the identity on $X$ and $H(\cdot,1)$ is a constant map.

The Hilbert cube $\mathcal{Q}$ is the topological product $\prod_{i=1}^{\infty}[-1,1].$ A \textit{Hilbert cube manifold} (or $\mathcal{Q}$-manifold) is a separable metrizable space that admits an open cover by sets homeomorphic to open subsets of $\mathcal{Q}.$ A fundamental result on Hilbert cube manifolds asserts that $\mathcal{Q}$ is the only $\mathcal{Q}$-manifold which is compact and contractible \cite[Theorem 7.5.8]{vanMillbook}. We refer to \cite{Chapmanbook,vanMillbook,Torunczyk1980} for the fundamentals of the theory of $\mathcal{Q}$-manifolds and infinite-dimensional topology. 

A closed subset $A$ of a metric space $X$ is a $Z$\textit{-set} if 
$\{f\in C(\mathcal{Q},X): f(\mathcal{Q})\cap A=\emptyset\}$ is dense in $C(\mathcal{Q},X).$ Here, $C(\mathcal{Q},X)$ is the space of continuous maps from $\mathcal{Q}$ to $X$ endowed with the compact-open topology.

\subsection{G-spaces}
\label{sec:Gspaces}

We recall some concepts from the theory of $G$-spaces that will be used throughout the work. By a $G$\textit{-space} we mean a Tychonoff space $X$ equipped with a continuous action $G\times X\rightarrow X$ of a   topological group $G$ on $X.$ The image under the action of the pair $(g,x)$ is denoted by $gx.$
If $S$ is a subset of a $G$-space $X$ and $H\subseteq G$ is a subgroup, then $H(S):=\{hs:h\in H, s\in S\}$ denotes the $H$-saturation of $S.$ In particular, $G(x)$ denotes the orbit of $x\in X.$ As usual, by $X/G$ we denote the orbit space. We say that a subset $S\subseteq X$ is $H$-invariant if $H(S)=S.$   

A continuous map $f:X\rightarrow Y$ between $G$-spaces is called \textit{equivariant} or simply a $G$-map if $f(gx)=gf(x),$ for any $g\in G$ and $x\in X.$  A $G$-map $r:X\rightarrow A,$ $A\subseteq X,$ is called $G$\textit{-strong deformation retraction}, if there is a homotopy $H:X\times[0,1]\rightarrow X$ such that: $H_0$ is the identity on $X;$ $H_1=r;$ $H(a,t)=a$ for each $t\in[0,1]$ and $a\in A;$ and $H_t$ is a $G$-map for every $t.$  In this context, $A$ is called $G$\textit{-strong deformation retract}. For a deeper discussion of $G$-spaces, we refer the reader to \cite{Bredon}.

A G-space $X$ is called \textit{proper} (in the sense of Palais \cite{Palais1961}), if it has an open cover such that
every $x\in X$ has a neighborhood $V$ for which the set 
$\{g\in G :gS\cap V\neq\emptyset\}$ has compact closure in $G$.
We recall the definition of a \textit{slice}.
Let $X$ be a $G$-space and $H$ a closed subgroup of $G.$ An $H$-invariant
subset $S\subseteq X$ is called an $H$-\textit{slice} in $X,$ if
$G\left(S\right)$ is open in $X$ and there exists a $G$-map $f:G\left(S\right)\rightarrow G/H$ such that $S=f^{-1}\left(eH\right).$
If $G\left(S\right)=X,$  $S$ is called a \textit{global
$H$-slice} of $X$ (see \textit{e.g.} \cite[p. 305]{Palais1}).

In this work, we are interested in $G$-spaces for which $G$ is a closed subgroup of the general lineal group $GL(d)$. Recall that $GL(d)$ consists of the linear isomomorphisms $T:\mathbb{R}^d\rightarrow\mathbb{R}^d.$ The topology on $GL(d)$ is the one induced by the operator norm 
$\|T\|=\text{sup}_{x\in B_2^d}\|T(x)\|.$ Here $B_2^d$ denotes the Euclidean ball associated to the standard Euclidean norm $\|\cdot\|$ on $\mathbb{R}^d.$ As usual, by $\mathbb{S}^{d-1}$ and $\langle\cdot,\cdot\rangle,$ we denote the  sphere $\partial B_2^d$ and the standard scalar product on $\mathbb{R}^d,$  respectively. 

Since we will deal with groups associated to tensor products of vector spaces, we briefly recall the following. The tensor product  $\otimes_{i=1}^{l}\mathbb{R}^{d_i}$ can be endowed with a Euclidean structure. This is the one determined by the Hilbert tensor product of Euclidean spaces, and it is induced by the scalar product $\langle\cdot,\cdot\rangle_H,$ 
\begin{equation}
\label{eq:euclideantensor}
 \left\langle x^{1}\otimes\cdots\otimes x^{l},y^{1}\otimes\cdots\otimes y^{l}\right\rangle _{H}:={\prod}_{i=1}^l\left\langle x^i,y^i\right\rangle.
\end{equation}
We simply write the definition of $\langle\cdot,\cdot\rangle_H$ on decomposable vectors $x^{1}\otimes\cdots\otimes x^{l}$ and $y^{1}\otimes\cdots\otimes y^{l}$ of $\otimes_{i=1}^{l}\mathbb{R}^{d_i}.$ However, it can be extended to the whole space by multilinearity \cite[Section 2.5]{kadisonkingrose}. The space $\otimes_{i=1}^{l}\mathbb{R}^{d_i}$ together with the scalar product $\langle\cdot,\cdot\rangle_H$ is denoted by $\otimes_{H,i=1}^{l}\mathbb{R}^{d_i}.$ Its corresponding norm and Euclidean ball are denoted by $\|\cdot\|_H$ and $B_2^{d_1,\dots,d_l},$ respectively. 
The set of decomposable vectors in $\otimes_{i=1}^l\mathbb{R}^{d_i}$ is denoted by $\Sigma_{d_1,\ldots,d_l}.$ That is,
$$
\Sigma_{d_1,\ldots,d_l}:=\{x^1\otimes\cdots\otimes x^l:x^i\in\mathbb{R}^{d_i},i=1,\ldots,l\}.
$$

By
$GL_\otimes(d_1,\ldots,d_l),$ we denote the group of linear isomorphisms on $\otimes_{H,i=1}^l\mathbb{R}^{d_i}$ sending decomposable vectors into decomposable vectors. $O_\otimes(d_1,\ldots,d_l)$ denotes the subgroup of $GL_\otimes(d_1,\ldots,d_l)$ consisting of orthogonal maps. In \cite[Section 4]{topologytensorialbodies}, it is proved that both $GL_\otimes(d_1,\ldots,d_l)$ and $O_\otimes(d_1,\ldots,d_l)$ are closed subgroups of $GL(\otimes_{i=1}^l\mathbb{R}^{d_i}).$ In particular, they are Lie groups and $O_\otimes(d_1,\ldots,d_l)$ is a maximal compact subgroup of $GL_\otimes(d_1,\ldots,d_l)$ \cite[Appendix A]{topologytensorialbodies}. To shorten notation, we simply write $GL_\otimes$ and $O_\otimes.$

Given $T\in GL_\otimes(d_1,\ldots,d_l),$ there exists $T_i\in GL(d_i),$ $i=1,\ldots,l,$ and a permutation  $\sigma$ on $\{1,\ldots,l\}$
 such that $T=(T_{1}\otimes\cdots\otimes T_{l})U_\sigma.$ Here, $U_\sigma\in GL_\otimes$ is defined in decomposable vectors as 
\begin{equation*}
U_\sigma(x^{1}\otimes\cdots\otimes x^{l})=x^{\sigma(1)}\otimes\cdots\otimes x^{\sigma(l)},
\end{equation*}
and $\sigma$ is such that $x^{\sigma(1)}\otimes\cdots\otimes x^{\sigma(l)}\in\otimes_{i=1}^l\mathbb{R}^{d_i}$ for  $x^i\in\mathbb{R}^{d_i},$ $i=1,\ldots,l.$
The subgroup of $GL_\otimes$ consisting of the maps $U_\sigma$ defined above is denoted by $\mathcal{P}.$  It is worth noticing that $\mathcal{P}\subset O_\otimes$. Indeed, the standard basis  $e_{k_i}^{d_i},$ $k_i=1,\ldots,d_i,$ of $\mathbb{R}^{d_i}$ is such that
 $U_{\sigma}(e_{k_1}^{d_1}\otimes\cdots\otimes e_{k_l}^{d_l})=e_{k_{\sigma(1)}}^{d_{\sigma(1)}}\otimes\cdots\otimes e_{k_{\sigma(l)}}^{d_{\sigma(l)}}.$ For a complete description of the properties of $\mathcal{P}$ and its relation with $GL_\otimes,$ we refer to \cite[Lemma A.1 and Proposition A.4]{topologytensorialbodies}. 

\section{Preliminaries on tensorial bodies}
\label{sec:basictensorialbodies}

We first introduce some results from convex geometry that we will be used in the work. The letters $d, d_i\geq2,$ with $i=1,\dots,l,$ will always denote natural numbers. A convex compact subset $P\subset\mathbb{R}^d$ with non-empty interior is called a \textit{convex body}. If in addition $P=-P,$ we say that $P$ is a $0$\textit{-symmetric convex body}. By $\mathcal{K}_d$ and $\mathcal{B}(d),$ we denote, respectively, the classes of convex bodies and $0$-symmetric convex bodies in $\mathbb{R}^d.$ The polar body of $P\in\mathcal{B}(d)$ is defined as $P^{\circ}=\{y\in\mathbb{R}^d:\langle x,y\rangle\leq1\text{ for all }x\in P\}.$ It is well-known that $P^{\circ}\in\mathcal{B}(d).$ The Minkowski functional associated to $P\in\mathcal{K}_d$ is defined by $g_P(x)=\text{inf}\{\lambda>0:\lambda^{-1}x\in P\},$ for every $x\in\mathbb{R}^d.$ In the case of $0$-symmetric convex bodies, the Minkowski functional provides a correspondence between norms on $\mathbb{R}^{d}$ and $0$-symmetric convex bodies. Indeed, the map sending $P\in\mathcal{B}(d)$ to its Minkowski functional $g_P(\cdot)$
is a bijection; the unit ball of $(\mathbb{R}^{d},g_P)$ is $P;$ and $g_{P^{\circ}}(x)=\left\Vert\langle \cdot,x\rangle\right\Vert,$ 
where $\langle \cdot,x\rangle$ is regarded as a linear functional from $\left(\mathbb{R}^{d},g_P\right)$
 to $\mathbb{R}.$ See \cite[Remark 1.7.8]{Schneider1993}.
The Hausdorff distance between non-empty compact sets $P,R\subset\mathbb{R}^{d}$ is defined as
\begin{equation*}
\delta^{H}(P,R):=\max\left\{ \underset{x\in P}{\sup}\left\{\underset{y\in R}{\inf}\| x-y\|\right\} ,\underset{y\in R}{\sup}\left\{\underset{x\in P}{\inf}\| y-x\|\right\} \right\},
\end{equation*}
or alternatively by $\delta^{H}(P,R)=\min\left\{ \lambda\geq0:P\subseteq R+\lambda B_{2}^d ,R\subseteq P+\lambda B_{2}^d \right\}$. For $P,R\in\mathcal{B}(d),$  the following is a well-known characterization of $\delta^{H}$: 
\begin{equation*}
\delta^{H}(P,R)=\underset{x\in\partial B_{2}^d} {\sup}\left|g_{P^{\circ}}\left(x\right)-g_{R^{\circ}}(x)\right|.
\end{equation*}
See \cite[Theorem 1.8.11]{Schneider1993}. Throughout the work $\mathcal{K}_d$ and $\mathcal{B}(d)$ will be regarded as metric spaces with respect to the Hausdorff distance. It is convenient to point out that in $\otimes_{i=1}^{l}\mathbb{R}^{d_i},$ the corresponding Hausdorff distance will be the one induced by the Euclidean norm $\|\cdot\|_H.$ In this case, the class of $0$-symmetric convex bodies in $\otimes_{i=1}^{l}\mathbb{R}^{d_i},$ denoted by $\mathcal{B}(\otimes_{i=1}^{l}\mathbb{R}^{d_i}),$ will be considered as a metric space with respect to the Hausdorff distance. 

For any $0$-symmetric convex body $P\subset\mathbb{R}^d$ or $P\subset\otimes_{i=1}^{l}\mathbb{R}^{d_i},$ the L\"{o}wner ellipsoid is denoted by $Low(P).$ Recall that $Low(P)$ is the ellipsoid of minimal volume containing $P$ \cite[Chapter 3]{Tomczak-Jaegermann1989}. We refer the reader to the monograph \cite{Schneider1993} for the basics about convex bodies.

\subsection{The space of tensorial bodies}
\label{sec:tensorialbodiesprev}

 By $\mathcal{B}_\otimes(d_1,\ldots,d_l)$ we denote the class of tensorial bodies in $\otimes_{i=1}^l\mathbb{R}^{d_i}.$ It consists of $0$-symmetric convex bodies $P\subset\otimes_{i=1}^l\mathbb{R}^{d_i}$ for which there exists $0$-symmetric convex bodies $P_i$ in the Euclidean space $\mathbb{R}^{d_i}$, $i=1,\ldots,l,$ such that
\begin{equation}
\label{eq:definition tensorbody}
P_{1}\otimes_{\pi}\cdots\otimes_{\pi}P_{l}\subseteq P\subseteq P_{1}\otimes_{\epsilon}\cdots\otimes_{\epsilon}P_{l}.
\end{equation}
Here, $\otimes_\pi$  is the projective tensor product of convex bodies \cite{Aubrun2006,maiteluisa2}
\begin{equation*}
P_{1}\otimes_{\pi}\cdots\otimes_{\pi}P_{l}:=\text{conv}\left\{ x^{1}\otimes\cdots\otimes x^{l}\in\otimes_{i=1}^{l}\mathbb{R}^{d_{i}}:x^{i}\in P_{i}, i=1,\ldots,l\right\},
\end{equation*}
and  $\otimes_\epsilon$ is the injective tensor product of convex bodies (with zero in the interior) \cite{Aubrun2017,maiteluisa2}
\begin{equation*}
P_{1}\otimes_{\epsilon}\cdots\otimes_{\epsilon}P_{l}:=(P_{1}^{\circ}\otimes_{\pi}\cdots\otimes_{\pi}P_{l}^{\circ})^{\circ}.
\end{equation*}
The polarity in $\otimes_{i=1}^l\mathbb{R}^{d_i}$ is the one determined by the scalar product  $\langle\cdot,\cdot\rangle_H,$ see (\ref{eq:euclideantensor}). Alternatively, the class of tensorial bodies  can be described in terms of the Minkowski functionals.  A $0$-symmetric convex body $P\subset\otimes_{i=1}^l\mathbb{R}^{d_i}$ is a tensorial body if and only if 
\begin{equation}
\label{eq:rc in gauges}
g_{P}\left(x^{1}\otimes\cdots\otimes x^{l}\right) = g_{P_{1}}\left(x^{1}\right)\cdots g_{P_{l}}\left(x^{l}\right) \text{ and }
\end{equation}
\begin{equation}
\label{eq:rc in gauges2}
g_{P^{\circ}}\left(x^{1}\otimes\cdots\otimes x^{l}\right) = g_{P_{1}^{\circ}}\left(x^{1}\right)\cdots g_{P_{l}^{\circ}}\left(x^{l}\right),
\end{equation}
for some $P_i\in\mathcal{B}(d_{i})$ and every $x^{i}\in\mathbb{R}^{d_{i}},$ $i=1,\ldots,l.$  Notice that, from the above equalities  and the bipolar theorem follow easily that $P$ is a tensorial body if and only if so is $P^\circ$. Moreover $\lambda P$ is a tensorial body, provided $\lambda\neq0$ and $P$ is a tensorial body \cite[Proposition 3.5]{tensorialbodies}.

The class of tensorial bodies was introduced in \cite{tensorialbodies}. It arises from the theory of norms on tensor products of Banach spaces. In this sense, in \cite[Proposition 3.1]{tensorialbodies}, it was proved that a norm $\alpha(\cdot)$ on the tensor product $\otimes_{i=1}^{l}(\mathbb{R}^{d_i},\|\cdot\|_i)$ is a reasonable crossnorm  if and only if its closed unit  ball $B_{(\otimes(\mathbb{R}^{d_i},\|\cdot\|_i),\alpha)}$ is such that
$$
B_{\|\cdot\|_1}\otimes_\pi\cdots\otimes_\pi B_{\|\cdot\|_l}\subseteq B_{(\otimes(\mathbb{R}^{d_i},\|\cdot\|_i),\alpha)}\subseteq B_{\|\cdot\|_1}\otimes_\epsilon\cdots\otimes_\epsilon B_{\|\cdot\|_l},
$$
where $B_{\|\cdot\|_i}\in\mathcal{B}(d_i)$ is the closed unit ball of the norm $\|\cdot\|_i,$ $i=1,\dots,l.$   In particular, $B_{(\otimes(\mathbb{R}^{d_i},\|\cdot\|_i),\alpha)}$ is a tensorial body in $\otimes_{i=1}^{l}\mathbb{R}^{d_i}$. The projective $\otimes_\pi$ and the injective $\otimes_\epsilon$ tensor products correspond with the so called projective $\pi(\cdot)$ and injective $\epsilon(\cdot)$ tensor norms, by means of the following relations:
\begin{equation}
\label{eq:relProjInjecnorm}
B_{\otimes_\pi(\mathbb{R}^{d_{i}},g_{P_{i}})}=P_{1}\otimes_{\pi}\cdots\otimes_{\pi}P_{l}\thinspace\text{ and } B_{\otimes_\epsilon(\mathbb{R}^{d_{i}},g_{P_{i}})}=P_{1}\otimes_{\epsilon}\cdots\otimes_{\epsilon}P_{l}. 
\end{equation}
Here $B$ denotes the closed unit ball of the projective and the injective tensor norms, and $P_i\subset \mathbb{R}^{d_{i}},$ $i=1,\dots,l,$ is an arbitrary $l$-tuple of $0$-symmetric convex bodies. In \cite[Theorem 3.2]{tensorialbodies}, by means of the Minkowski functional, a bijection between $\mathcal{B}_\otimes(d_1,\ldots,d_l)$ and the class of reasonable crossnorms on finite dimensions is exhibited. In consequence, the tensorial bodies are, precisely, the closed unit balls of finite-dimensional reasonable crossnorms. For the notation and a fuller treatment of tensor norms, we refer the reader to the classical monographs \cite{defantfloret,Ryan2013}. 
For a detailed discussion about the interplay between tensorial bodies and tensor norms, we refer to \cite{tensorialbodies, maiteluisa2}.

The class of tensorial bodies $\mathcal{B}_\otimes(d_1,\ldots,d_l)$ is contained in $\mathcal{B}(\otimes_{i=1}^{l}\mathbb{R}^{d_i}),$ and thus it is also a metric space with respect to the Hausdorff distance. We refer to it as \textit{the space of tensorial bodies}. Indeed, it is a $GL_\otimes$-space \cite[Section 4]{topologytensorialbodies}.  The action is given by the evaluation $T(P),$ for $T\in GL_\otimes$ and $P\in\mathcal{B}_\otimes(d_1,\ldots,d_l).$ In particular, for  
the projective $\otimes_{\pi}$ and the injective  $\otimes_{\epsilon}$ tensor product, 
the following holds: 
\begin{equation}
\label{eq:gltensor image of proj and inj tp}
T(P_{1}\otimes_{\alpha}\cdots\otimes_{\alpha}P_{l})= T_{1}(P_{\sigma(1)})\otimes_{\alpha}\cdots\otimes_{\alpha}T_{l}(P_{\sigma(l)}),\text{ }\alpha=\pi,\epsilon,
\end{equation}
for every tuple of $0$-symmetric convex bodies $P_i\subset\mathbb{R}^{d_i},$ $i=1,\ldots,l.$ In this case, we have used the representation of $T\in GL_\otimes$ as $T=(T_{1}\otimes\cdots\otimes T_{l})U_\sigma$ for some $T_i\in GL(d_i),$ $i=1,\dots,l,$ and $U_\sigma\in\mathcal{P}$ (see Section \ref{sec:Gspaces}).
Many topological properties of this space were exhibited in \cite{topologytensorialbodies}, below we summarize  those which are relevant for this work.

For every fixed tuple $P_i\in\mathcal{B}(d_i),$ $i=1,\dots,l,$ the hyperspace $\mathcal{B}_{P_1,\dots,P_l}(d_1,\ldots,d_l)$ consists of the tensorial bodies for which (\ref{eq:definition tensorbody}) hold. If $P\in\mathcal{B}_{P_1,\dots,P_l}(d_1,\ldots,d_l),$ we say that $P$ is a \textit{tensorial body with respect to} $P_1,\dots,P_l.$ 
As a consequence of \cite[Corollary 3.4]{tensorialbodies}, for every tensorial body $P$ in $\otimes_{i=1}^{l}\mathbb{R}^{d_i},$ one can construct explicitly different $l$-tuples $(P_1,\ldots,P_l)\in \mathcal{B}(d_1)\times\cdots\times\mathcal{B}(d_l)$ for which $P\in\mathcal{B}_{P_1,\dots,P_l}(d_1,\ldots,d_l).$ In this sense, to simplify the arguments in the forthcoming proofs, it is convenient to fix one of these $l$-tuples associated to each $P\in\mathcal{B}_\otimes(d_1,\ldots,d_l).$
\begin{remark}
\label{rem:Q-ifixed}
Let $P\subset\otimes_{i=1}^{l}\mathbb{R}^{d_i}$ be a tensorial body and let $P^i\in\mathcal{B}(d_i),$ $i=1,\ldots,l,$ be defined as 
\begin{equation*}\label{eq: Q1 Ql}
P^{i}:=\left\{ x^{i}\in\mathbb{R}^{d_{i}}:\frac{e^{d_1}_{1}\otimes\cdots\otimes e^{d_{i-1}}_{1}\otimes x^i\otimes e^{d_{i+1}}_{1}\otimes\cdots\otimes e^{d_l}_{1}}{g_{P}(e^{d_1}_{1}\otimes e^{d_2}_{1}\otimes\cdots\otimes e^{d_l}_{1})}\in P\right\},
\end{equation*}
for $i=1,\dots,l-1,$ and for $i=l$ as $P^l:=\left\{x^l\in\mathbb{R}^{d_l}:e^{d_1}_{1}\otimes\cdots\otimes e^{d_{l-1}}_{1}\otimes x^l\in P\right\}.$ It is straightforward to check that  $P^i$ is precisely the $0$-symmetric convex body $P_i^{e_1^{d_1},e_1^{d_2},\ldots,\lambda e_1^{d_l}},$ $i=1,\ldots,l,$ $\lambda=\frac{1}{g_P(e^{d_1}_{1}\otimes e^{d_2}_{1}\otimes\cdots\otimes e^{d_l}_{1})},$ defined in \cite[(9)]{tensorialbodies}. Thus, from  \cite[Corollary 3.4]{tensorialbodies}, it  follows directly that $P$ is a tensorial body with respect to $P^1,\dots,P^l.$   
\end{remark}

By $\mathcal{E}_\otimes(d_1,\ldots,d_l),$ we denote the class of $0$-symmetric ellipsoids in $\otimes_{i=1}^l\mathbb{R}^{d_i}$ which also are tensorial bodies. Recall that $E\subset\otimes_{i=1}^{l}\mathbb{R}^{d_i}$ is a $0$-symmetric ellipsoid (not necessarily tensorial) if there exists a linear isomorphism $T$ on $\otimes_{i=1}^{l}\mathbb{R}^{d_i}$ such that $E=T(B_2^{d_1,\dots,d_l}).$ We usually write $\mathcal{E}_\otimes$ instead of $\mathcal{E}_\otimes(d_1,\ldots,d_l).$
\begin{proposition}
\label{prop:resumeprop}
\cite{topologytensorialbodies} The following hold:
\begin{enumerate}
    \item The tensor products $\otimes_\pi,$ $\otimes_\epsilon$: $\mathcal{B}(d_1)\times\cdots\times\mathcal{B}(d_l)\rightarrow\mathcal{B}_\otimes(d_1,\ldots,d_l)$ are continuous with respect to the Hausdorff distance.
    
    \item The map $conv_\otimes:\mathcal{B}_\otimes(d_1,\ldots,d_l)\rightarrow\mathcal{B}_\otimes(d_1,\ldots,d_l),$ defined as 
   \begin{equation}
\label{eq:tensorialconvexhull}
    conv_\otimes(P):=conv(P\cap\Sigma_{d_1,\ldots,d_l}),
   \end{equation}
    is a $GL_\otimes$-equivariant retraction onto $
    \Pi(d_1,\dots,d_l):=\{P_{1}\otimes_{\pi}\cdots\otimes_{\pi}P_{l}:P_i\in\mathcal{B}(d_i)\}.$ 
    
    \item The map $\ell_\otimes=Low\circ conv_\otimes$ sending $P$ to 
    $Low(conv_\otimes(P))$ is a $GL_\otimes$-equivariant retraction of  $\mathcal{B}_\otimes(d_1,\ldots,d_l)$ onto  $\mathcal{E}_\otimes(d_1,\ldots,d_l).$
\end{enumerate}
\end{proposition}

\begin{theorem}
\label{thm:resumenthm}
\cite{topologytensorialbodies} Let $p=\frac{d_1(d_1+1)}{2}+\cdots+\frac{d_l(d_l+1)}{2},$ with $d_i\geq2$ for $i=1,\dots,l.$ The following hold:
\begin{enumerate}
\item $\mathcal{B}_\otimes(d_1,\ldots,d_l)$ is a closed and contractible subspace of  
$\mathcal{B}(\otimes_{i=1}^l\mathbb{R}^{d_i}).$ 

\item $\mathcal{B}_\otimes(d_1,\ldots,d_l)$ is a proper $GL_\otimes$-space.

\item  $\mathcal{E}_\otimes(d_1,\ldots,d_l)$  is the $GL_\otimes$-orbit of $B_2^{d_1,\dots,d_l}.$ Furthermore, it is homeomorphic to $\mathbb{R}^p.$

\item The set  $\mathcal{L}_{\otimes}(d_1,\ldots,d_l),$ defined as $$\mathcal{L}_{\otimes}(d_1,\ldots,d_l):=\{ P\in\mathcal{B}_{\otimes}(d_1,\ldots,d_l):\ell_{\otimes}(P)=B_{2}^{d_{1},\ldots,d_{l}}\},$$ is a compact $O_\otimes$-global slice for the  $GL_\otimes$-space $\mathcal{B}_\otimes(d_1,\ldots,d_l).$ Moreover, the spaces $\mathcal{B}_\otimes(d_1,\ldots,d_l)/GL_\otimes(d_1,\ldots,d_l)$ and $\mathcal{L}_\otimes(d_1,\ldots,d_l)/O_\otimes(d_1,\ldots,d_l)$ are homeomorphic.

\item The orbit space $\mathcal{B}_\otimes(d_1,\ldots,d_l)/GL_\otimes(d_1,\ldots,d_l)$ is compact and metrizable. Indeed, the logarithm $\text{ log}\delta_{\otimes}^{BM}$ of \textit{the tensorial Banach-Mazur distance} $\delta_{\otimes}^{BM}$ is a metric on it. \textit{The tensorial Banach-Mazur distance} $\delta_{\otimes}^{BM}$ is defined as 
\begin{equation}
\label{eq:tensorial bm distance}
\delta_{\otimes}^{BM}\left(P,R\right):=\inf\left\{ \lambda\geq1:R\subseteq TP\subseteq\lambda R,\text{ for }T\in GL_{\otimes}(d_1,\dots,d_l)\right\},
\end{equation}
for $P,R\in\mathcal{B}_\otimes(d_1,\ldots,d_l).$
\end{enumerate}

\end{theorem}
By $\mathcal{E}(d_i),$ we denote the class of $0$-symmetric ellipsoids in $\mathbb{R}^{d_i},$ $i=1,\ldots,l.$ Given any $l$-tuple of $0$-symmetric ellipsoids $E_i=T_i(B_2^{d_i}),$ $T_i\in GL(d_i)$ and $i=1,\dots,l,$ their Hilbertian tensor product $E_1\otimes_2\cdots\otimes_2 E_l$ is a tensorial body in $\otimes_{i=1}^l\mathbb{R}^{d_i}$ defined as
$$
E_1\otimes_2\cdots\otimes_2 E_l:=T_1\otimes\cdots\otimes T_l(B_{2}^{d_{1},\ldots,d_{l}}).
$$ 
It is important to notice that this definition does not depend on the selection of the maps $T_i.$ See \cite{Aubrun2006} or \cite[Section 4]{tensorialbodies} for the details and the basic properties of $\otimes_2$. 
In \cite[Corollary 4.3]{tensorialbodies}, it is proved that the class of tensorial ellipsoids $\mathcal{E}_\otimes(d_1,\ldots,d_l)$ consists only of tensor products $E_1\otimes_2\cdots\otimes_2 E_l$ of ellipsoids $E_i\in\mathcal{E}(d_i),$ $i=1,\ldots,l$. Actually, the Euclidean ball $B_2^{d_1,\dots,d_l}$ is a tensorial ellipsoid, since $B_2^{d_1,\dots,d_l}=I_1\otimes\cdots\otimes I_l(B_2^{d_1,\dots,d_l}),$ where $I_i$  is the identity map on $\mathbb{R}^{d_i},$ $i=1,\dots,l.$ In particular, it is a tensorial body with respect to the Euclidean balls $B_2^{d_1},\dots,B_2^{d_l}.$

\begin{proposition}
\label{prop:continuity of hilbert}
The tensor product $\otimes_2:\mathcal{E}(d_1)\times\cdots\times\mathcal{E}(d_l)\rightarrow\mathcal{E}_\otimes(d_1,\ldots,d_l)$ is continuous with respect to the Hausdorff distance.
\end{proposition}

\begin{proof}
Recall that the topology on $GL(d_i)$ is the one inherited from the space of linear maps on the Euclidean space $\mathbb{R}^{d_i},$ similarly the topology on $GL_\otimes$ is the one determined by the space of linear maps on $\otimes_{H,i=1}^l\mathbb{R}^{d_i}.$ We will need the topological representative of $\mathcal{E}(d_i)$ exhibited in the proof of \cite[Corollary 3.9]{AntonyanNatalia}. There,  $\mathcal{E}(d_i)$ is represented as the closed subset 
$\mathcal{A}_i\subset GL(d_i)$ consisiting of the positive self-adjoint maps $T_i:\mathbb{R}^{d_i}\rightarrow\mathbb{R}^{d_i}.$ The homeomorphism is given by the map $\xi_i:\mathcal{E}(d_i)\rightarrow\mathcal{A}_i$ sending $E_i\in\mathcal{E}(d_i)$ to the unique linear map $\xi_i(E_i)\in\mathcal{A}_i$ such that $\xi_i(E_i)B_2^{d_i}=E_i$, see  \cite[(3.4)]{AntonyanNatalia}. 
If we let $F:\mathcal{A}_1\times\cdots\times\mathcal{A}_l\rightarrow\mathcal{E}_\otimes(d_1,\ldots,d_l)$ be the map
$$
F(T_1,\ldots,T_l)=T_1\otimes\cdots\otimes T_l(B_{2}^{d_{1},\ldots,d_{l}}),
$$
then $\otimes_2=F\circ(\xi_1,\ldots,\xi_l).$ Thus, for any $S_i,T_i\in\mathcal{A}_i$ 
$$
\delta^H(F(T_1,\ldots,T_l),F(S_1,\ldots,S_l))\leq\|T_1\otimes\cdots\otimes T_l-S_1\otimes\cdots\otimes S_l\|,
$$
where $\|T_1\otimes\cdots\otimes T_l-S_1\otimes\cdots\otimes S_l\|$ denotes the usual norm of linear maps on the Euclidean space $\otimes_{H,i=1}^{l}\mathbb{R}^{d_i}.$
This shows that $F$ is a continuous map and so the same holds for
$\otimes_2$. 
\end{proof}
The following lemmas will be used later in the work. Our proof of Lemma \ref{lem:infinitedimension} follows along the same ideas of \cite[Theorem 2.2]{nadler1979}, we have included it for the sake of transparency.
\begin{lemma}
\label{lem:continuous nu}
The map $\nu:\mathcal{B}(d)\times\mathcal{B}(d)\rightarrow(0,\infty),$ $d\geq2,$ defined as $\nu(P,Q)=\sup_{z\in P}g_Q(z)$ is continuous.
\end{lemma}
\begin{proof}
It can be easily checked that if $P,R,\bar{P},\bar{R}\in\mathcal{B}(d),$ then
\begin{align*}
|\nu(P,R)-\nu(\bar{P},\bar{R})|&\leq|\nu(P,R)-\nu(P,\bar{R})|+
|\nu(P,\bar{R})-\nu(\bar{P},\bar{R})|\\
&\leq\sup_{z\in P}|g_{R}(z)-g_{\bar{R}}(z)|+\delta^H(P,\bar{P})\nu(B_2^d,\bar{R})
\end{align*}
Now, for every pair of sequences  $\{P_n\},\{R_n\}\subset\mathcal{B}(d)$ converging to $P,R$ respectively, we have that $\{\nu(B_2^d,R_n)\}_n$ is bounded and that $\{g_{R_n}(\cdot)\}_n$ converges uniformly on $P$ to $g_{R}(\cdot)$ (\cite[Lemma 3.1]{topologytensorialbodies}). Thus, from the above inequalities, it follows that $\nu(P_n,R_n)$ approaches to $\nu(P,R)$ and so the map $\nu$ is continuous.
\end{proof}

\begin{lemma}
\label{lem:infinitedimension}
Let $A$ and $K$ be $0$-symmetric convex bodies in $\mathbb{R}^d$ with $A\subsetneqq K.$ Then the hyperspace $csb_{A}(K):=\{P\in\mathcal{B}(d):A\subseteq P\subseteq K\}$ is homeomorphic to $\mathcal{Q}$.
\end{lemma}

\begin{proof}

First, let us observe that  $csb_A(K)$ is compact. Indeed, by the Blaschke selection theorem, any sequence $\{R_n\}_n\subset csb_{A}(K)$ has a subsequence $\{R_{n_i}\}$ converging to a compact convex set $R.$ Clearly, $R$ must be $0$-symmetric and $A\subseteq R\subseteq K.$ Therefore $R\in csb_{A}(K)$ and the hyperspace is compact.

On the other hand,  the support function provides an affine embedding of $csb_{A}(K)$ into $C(\mathbb{S}^{d-1}),$ the space of continuous functions on the sphere $\mathbb{S}^{d-1}=\partial B_2^d.$   By the properties of the Minkowski sum, the image of $csb_{A}(K)$ under this embedding is convex. 

Hence, $csb_{A}(K)$ is homeomorphic to a compact convex subset of a Banach space. Therefore, by \cite[Theorem 7.1]{BessagaPelczynski}, $csb_{A}(K)$ is homeomorphic to either $\mathcal{Q}$ or to $[0,1]^n$  for some natural number $n$.
  Since every cube $[0,1]^n$ has covering dimension equal to $n$,  in order to complete the proof it is enough to show that  $csb_{A}(K)$ does not have finite covering dimension.
To that end, we will show that $csb_{A}(K)$ contains an $n$-cell for every $n\geq 1,$ the result then will follow from \cite[Ch. 3, Proposition 1.5]{PearsAR}).

Since $A\neq K,$ there exists $x_0\in\text{int}(K)\setminus A.$ Let $r_0>0$ be such that the closed unit ball  $B[x_0,r_0]:=x_0+r_0B_2^d\subset\text{int}(K)\setminus A$ and $2r_0B_2^d\subset A$. Let us denote by $p(x_0,A)\in A$  the point such that $\text{dist}(x_0,A)=\|x_0-p(x_0,A)\|,$ and let $$u(x_0,A)=\frac{x_0-p(x_0,A)}{\text{dist}(x_0,A)}.$$
 It is known that $H,$  the hyperplane through $p(x_0,A)$ orthogonal to $u(x_0,A),$ supports $A$ (\cite[Lemma 1.3.1]{Schneider1993}. 
Consider any fixed  $v\in H$ with $v\neq p(x_0,A),$ 
and let $\mathcal{V}$ be the 2-dimensional plane through $p(x_0,A)$ generated by $v-p(x_0,A)$ and $u(x_0,A).$ Namely, $$\mathcal{V}=\{p(x_0,A)+s(v-p(x_0,A))+tu(x_0,A): s,t\in\mathbb{R}\}.$$
Notice that for every $0<r<r_0,$ the 2-dimensional closed disk $D=B[x_0,r]\cap\mathcal{V}\subset K$  is separated from the section $A\cap \mathcal{V}$ by $H\cap\mathcal{V}.$ Let $y_N=x_0+ru(x_0,A)\in D.$ 

\textbf{\textit{Claim 1}.} There exist $\varepsilon>0$  such that for any pair of points $z, y\in B[y_N,\varepsilon]\cap \partial D$, the line determined by $z$ and $y$ does not intersect $A$.

Suppose that the opposite occurs, then for every $\varepsilon>0,$ there exists $y_{\varepsilon}, z_{\varepsilon}\in B[y_N,\varepsilon]\cap \partial D$ for which $L_\varepsilon\cap A\neq\emptyset,$  where $L_\varepsilon$ denotes the line determined by $z_{\varepsilon}$ and $y_{\varepsilon}$. Pick a point $a(\varepsilon)\in L_\varepsilon\cap A$, and let $L$ be the line in $\mathcal V$ which is tangent to $D$ at the point $y_N$. Also, denote by $\Lambda$ the line generated by $p(x_0,A)$ and $v$ (namely, $\Lambda=H\cap \mathcal V)$. Observe that  $L$ and $\Lambda$ are parallel lines and both are completely contained in the plane $\mathcal V$.  Since $L_{\varepsilon}$ intersects $A\cap \mathcal V$, and $A\cap \mathcal V$ is contained in the half plane determined by $\Lambda$ that does not contain the pair $\{y_\varepsilon, z_\varepsilon\}$, we also have that $L_\varepsilon\cap \Lambda\neq \emptyset$. Let us denote by $b(\varepsilon)$ the intersection point of these two lines.
Consider the triangle in $\mathcal{V}$ with vertices $y_{\varepsilon},$ $p(x_0,A)$ and $a(\varepsilon)$. Observe that $b(\varepsilon)$ lies in the segment with extremes $y_{\varepsilon}$ and $a(\varepsilon)$ and therefore
\begin{align*}
     diam(A)&\geq \|a(\varepsilon)-p(x_0,A)\|
    \geq \|a(\varepsilon)-y_\varepsilon\|-\|y_{\varepsilon}-p(x_0,A)\|\\
    &\geq \|a(\varepsilon)-y_\varepsilon\|-(\|y_{\varepsilon}-x_0\|+\|x_0-p(x_0,A)\|) \\
       &\geq \|b(\varepsilon)-y_\varepsilon\|-(r+\text{dist}(x_0,A)) 
\end{align*}
Notice that if $\varepsilon$ approaches to $0$, then $\|b(\varepsilon)-y_\varepsilon\|$ diverges to $\infty$.
This contradicts the fact that $A$ is compact and therefore the claim is proved.



Let $\varepsilon\in (0, r)$ be as in Claim 1, and let us denote by $T$  the planar region determined by the arc $\partial D\cap B[y_N, \varepsilon]$ and the chord through its extreme points. 
Then, for any integer $n\geq3,$ there exists an integer $m=m(n)>n$ such that for some regular $m$-sided polygon $P_{m}$ with sides $S_1,\ldots,S_{m}$ inscribed in $D,$  there are (at least) $n$  consecutive sides contained in $T.$ Let us suppose, without loss of generality, that the sides $S_1,\ldots,S_{n}$ are contained in $T.$ Next, for every $j=1,\ldots,n,$ let $C_j$ be the segment $[b_j,c_j]:=\{tc_j+(1-t)b_j:t\in[0,1]\},$ where $b_j$ is the middle point of $S_j,$ and $c_j$ is the intersection of $\partial D$ with the ray $R_j=\{sb_j+(1-s)x_0:s\geq1\}$ (see Figure \ref{fig:auximage}). 
By our selection of $T,$ if $j=1,\ldots,n-1,$ $z\in C_{j}$ and $w\in C_{j+1},$ then the line $\mathcal{W}=\{z+t(z-w):t\in\mathbb{R}\}$ is completely contained in $\mathcal{V}\setminus A.$ 
\begin{figure}[h]
\scalebox{0.2}{\includegraphics{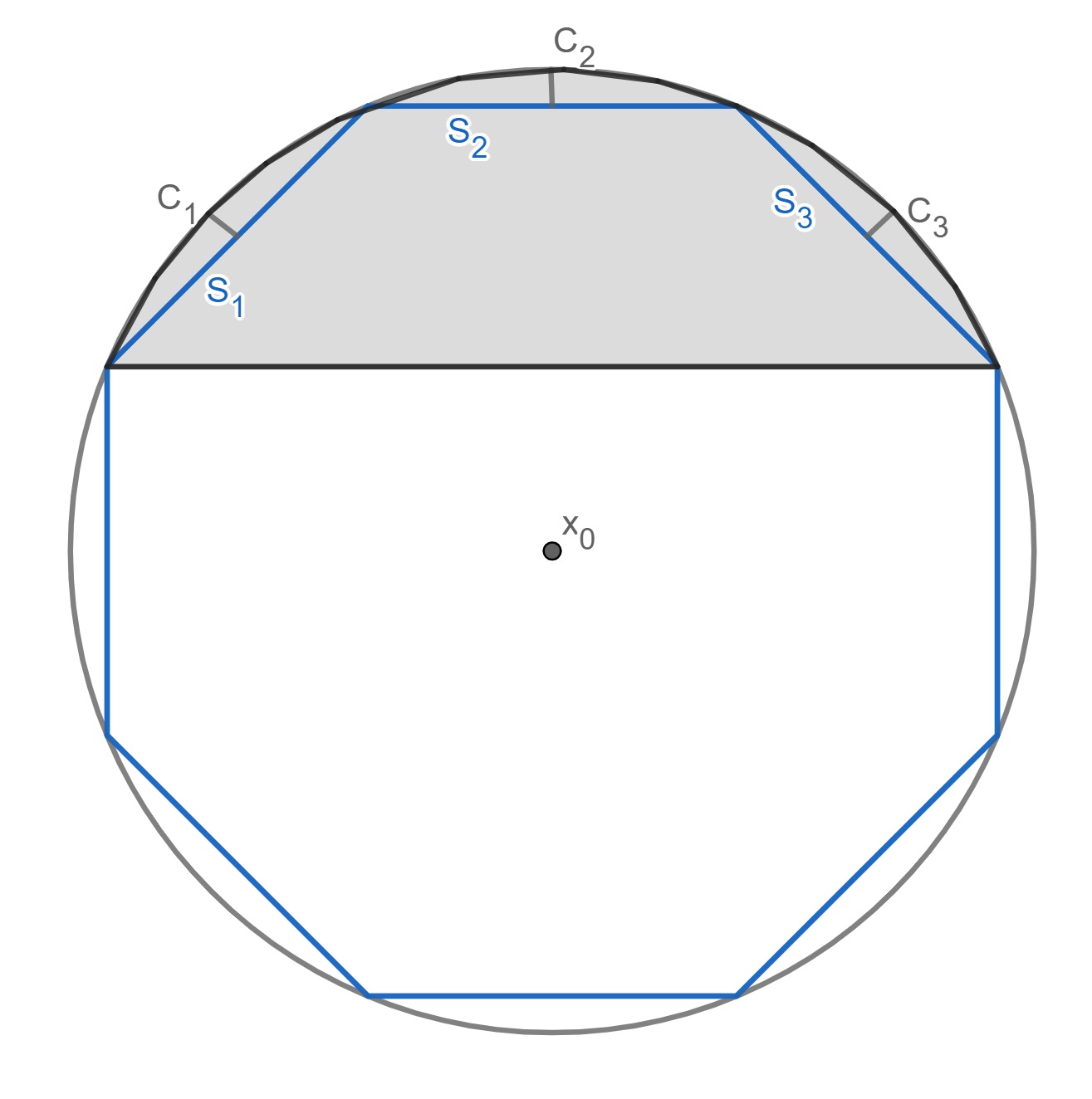}}
\centering
\caption{Construction of Lemma \ref{lem:infinitedimension}. Here, the sides  $S_{1},S_{2},S_{3}$ of the  polygon $P_8$   and segments $C_1,C_2,C_3$ are contained in the shaded region $T.$}
\label{fig:auximage}
\end{figure}

\textbf{\textit{Claim 2}.} The map $G:C_{1}\times\cdots\times C_{n}\rightarrow csb_{A}(K),$ sending the tuple $(z_1,\ldots,z_n)$ to $\text{conv}(\{\pm z_1,\ldots,\pm z_n\}\cup A)$ is a homeomorphism of the $n$-cell $C_{1}\times\cdots\times C_{n}$ into $csb_{A}(K).$ 

Clearly, $G(z_1,\ldots,z_n)\subseteq K$ is a $0$-symmetric convex body containing $A.$ In addition, it is not difficult to check that
$$
\delta^H(G(z_1,\ldots,z_n),G(w_1,\ldots,w_n))\leq\sum_{j=1}^{n}\|z_i-w_i\|.
$$
Hence, $G$ is a continuous map. Moreover, since $G$ is defined on a compact set, it must be closed.  To prove that $G$ is one-to-one, notice that if $G(z_1,\ldots,z_n)=G(w_1,\ldots,w_n),$ then $G(z_1,\ldots,z_n)\cap\mathcal{V}=G(w_1,\ldots,w_n)\cap\mathcal{V}$ and therefore their extreme points coincide.
The  fact that $2rB_2^d\subset 2r_0B_2^d\subset A$, in  combination with our construction, guarantees that $\{z_1,\dots, z_m\}$ and $\{w_1,\dots,w_m\}$ are  extreme points of $G(z_1,\ldots,z_n)\cap\mathcal{V}$ and $G(w_1,\ldots,w_n)\cap\mathcal{V}$, respectively. Thus, $z_j=w_j$ for every $j=\{1,\dots, n\}$ and therefore  $G$ is a homeomorphism, as desired.

 Notice that Claim 2 implies that  $csb_{A}(K)$ contains a $n$-cell, for every $n\geq3$. This proves that $csb_{A}(K)$ does not have finite covering dimension and therefore the proof is complete. 


\end{proof}

\section{A hyperspace associated to $\otimes_\pi$}
\label{sec:relatedhyperspaces}

In this section, we calculate the topological structure of the space $\Pi(d_1,\dots,d_l)$ consisting of the tensor products $P_{1}\otimes_{\pi}\cdots\otimes_{\pi}P_{l}$ of $0$-symmetric convex bodies $P_i\subset\mathbb{R}^{d_i},$ $i=1,\dots,l.$ To this end, we go further into the properties of the $GL_\otimes$-map  $conv_\otimes:\mathcal{B}_\otimes(d_1,\dots,d_l)\rightarrow\Pi(d_1,\dots,d_l),$ defined in Proposition \ref{prop:resumeprop}.

Let us denote by $H$  the closed subgroup of $GL(d_1)\times\cdots\times GL(d_l)$ defined as  
$$
H:=\{(\lambda_1I_{d_1},\ldots,\lambda_lI_{d_l}): \lambda_1\cdots\lambda_l=1 \text{ and each }\lambda_i>0 \}.
$$ 
Here $I_{d_i}$ denotes the identity map of $\mathbb{R}^{d_i},$ $i=1,\dots,l.$ If we consider the natural action of each $GL(d_i)$ on $\mathcal{B}(d_i),$ then 
the evaluation $(T_1(P_1),\ldots,T_l(P_l)),$ $T_i\in GL(d_i)$ and $P_i\in\mathcal{B}(d_i),$ $i=1,\dots,l,$ defines a continuous action of $GL(d_1)\times\cdots\times GL(d_l)$ on $\mathcal{B}(d_1)\times\cdots\times\mathcal{B}(d_l).$ As a consequence, its restriction to  $H$ determines a continuous action of $H$ on $\mathcal{B}(d_1)\times\cdots\times\mathcal{B}(d_l).$ 
In a similar way, in the subspace $\mathcal{E}(d_1)\times\cdots\times\mathcal{E}(d_l),$ consisting of $l$-tuples of $0$-symmetric ellipsoids, the restriction of the action on $\mathcal{B}(d_1)\times\cdots\times\mathcal{B}(d_l)$ also induces a  natural action of $H$ on $\mathcal{E}(d_1)\times\cdots\times\mathcal{E}(d_l).$

We now exhibit the relation between the $H$-orbits of the  $l$-tuples  $(P_1,\dots,P_l)\in\mathcal{B}(d_1)\times\cdots\times\mathcal{B}(d_l)$ and $\otimes_\pi.$ 
By \cite[Proposition 5]{maiteluisa2}, for every  $l$-tuple $P_i\in\mathcal{B}(d_i),$ $i=1,\dots,l,$ we have that 
 \begin{align}
\label{eq:projectivelambda}
\lambda_1P_1\otimes_\pi\cdots\otimes_\pi\lambda_lP_l=(\lambda_1\cdots\lambda_l)P_1\otimes_\pi\cdots\otimes_\pi P_l
=P_1\otimes_\pi\cdots\otimes_\pi P_l,
\end{align}
as long as each $\lambda_i>0$ and $\lambda_1\cdots\lambda_l=1.$ 
Hence, $\lambda_1P_1\otimes_\pi\cdots\otimes_\pi\lambda_lP_l$ belongs to $\otimes_\pi^{-1}(P_{1}\otimes_{\pi}\cdots\otimes_{\pi}P_{l}).$ Conversely, if $(R_1,\ldots,R_l)\in\otimes_\pi^{-1}(P_{1}\otimes_{\pi}\cdots\otimes_{\pi}P_{l}),$ then there exists real numbers $\lambda_i>0$ such that $\lambda_1\cdots\lambda_l=1$ and $R_i=\lambda_iP_i,$ $i=1,\ldots,l$ \cite[Proposition 3.6]{tensorialbodies}.
These inclusions prove that
\begin{equation}
\label{eq:inverseimage}
\otimes_\pi^{-1}(P_{1}\otimes_{\pi}\cdots\otimes_{\pi}P_{l})=\{(\lambda_1P_1,\dots,\lambda_lP_l):\lambda_1\cdots\lambda_l=1 \text{ and each }\lambda_i>0\}.
\end{equation}
In consequence,  $\otimes_\pi^{-1}(P_{1}\otimes_{\pi}\cdots\otimes_{\pi}P_{l})$ is the $H$-orbit of each $l$-tuple  $(P_1,\dots,P_l)\in\mathcal{B}(d_1)\times\cdots\times\mathcal{B}(d_l).$  
To shorten notation, we usually write $\mathcal{B}_\otimes$ and $\Pi$ instead of $\mathcal{B}_\otimes(d_1,\ldots,d_l)$ and $\Pi(d_1,\ldots,d_l).$

\begin{proposition} The following hold:
\label{prop:convisproper}
\begin{enumerate}
\item The $GL_\otimes$-equivariant retraction $conv_\otimes$ is proper. 
\item $\Pi(d_1,\ldots,d_l)$ is a $GL_\otimes$-strong deformation retract of $\mathcal{B}_\otimes(d_1,\ldots,d_l).$
\end{enumerate}
\end{proposition}
\begin{proof}
1. We show that for any compact subset $Y\subseteq\Pi,$  $conv_\otimes^{-1}(Y)$ is compact too.  To this end, let $\{P_n\}_n$ be a sequence in $conv_\otimes^{-1}(Y)$,  we will see that it contains a convergent subsequence. By the compactness of $Y,$ there exists a subsequence $\{P_{n_k}\}$ such that $conv_\otimes(P_{n_k})$ converges to some $P\in Y.$  Now, from the triangle inequality we have
\begin{equation}
\label{eq:convtensorpropia0}
\delta^H(P_{n_k},P)\leq\delta^H(P_{n_k},conv_\otimes(P_{n_k}))+\delta^H(conv_\otimes(P_{n_k}),P).
\end{equation}
On the other hand, if $P_{n_k}^i,$ $i=1,\dots,l,$ are the convex bodies associated to $P_{n_k}$ in Remark \ref{rem:Q-ifixed}, then $P_{n_k}$ and $conv_\otimes(P_{n_k})$ belong to $\mathcal{B}_{P_{n_k}^1,\dots,P_{n_k}^l}(d_1,\dots,d_l),$ 
and, from (\ref{eq:definition tensorbody}), we have 
\begin{align}
\label{eq:convtensorpropia1}
\delta^H(P_{n_k},conv_\otimes(P_{n_k}))\leq\delta^H(P_{n_k}^1\otimes_\pi\cdots\otimes_\pi P_{n_k}^l,P_{n_k}^1\otimes_\epsilon\cdots\otimes_\epsilon P_{n_k}^l).
\end{align}
 Additionally,  from the relations between the projective and the injective norm exhibited in \cite[Proposition 2.4]{maite} combined with the identities in (\ref{eq:relProjInjecnorm}), we get
$$
P_{n_k}^1\otimes_\pi\cdots\otimes_\pi P_{n_k}^l\subset P_{n_k}^1\otimes_\epsilon\cdots\otimes_\epsilon P_{n_k}^l\subset\left(\prod_{i=1}^{l-1}d_i\right)P_{n_k}^1\otimes_\pi\cdots\otimes_\pi P_{n_k}^l,
$$
and so 
\begin{align}
\label{eq:convtensorpropia2}
\delta^H(P_{n_k}^1\otimes_\pi\cdots\otimes_\pi P_{n_k}^l,&P_{n_k}^1\otimes_\epsilon\cdots\otimes_\epsilon P_{n_k}^l)\leq\left(\prod_{i=1}^{l-1}d_i-1\right)\nu(P_{n_k}^1\otimes_\pi\cdots\otimes_\pi P_{n_k}^l,B_2^{d_1,\dots,d_l})\nonumber\\
&\overset{*}{=}\left(\prod_{i=1}^{l-1}d_i-1\right)\nu(conv_\otimes(P_{n_k}),B_2^{d_1,\dots,d_l}).
\end{align}
(*) Folllows from the fact that $conv_\otimes$ is constant on $\mathcal{B}_{P_{n_k}^1,\dots,P_{n_k}^l}(d_1,\dots,d_l)$  \cite[Proposition 4.4]{topologytensorialbodies}, and thus
$conv_\otimes(P_{n_k})=P_{n_k}^1\otimes_\pi\cdots\otimes_\pi P_{n_k}^l.$

Therefore, from the convergence of $conv_\otimes(P_{n_k})$ to $P$, the continuity of $\nu$ (Lemma \ref{lem:continuous nu}), and the inequalities (\ref{eq:convtensorpropia0}), (\ref{eq:convtensorpropia1}) and (\ref{eq:convtensorpropia2}), it follows that $\{P_{n_k}\}_{n_k}$ must be bounded. Hence, by the Blaschke selection theorem \cite[Theorem 1.8.6]{Schneider1993},  we may suppose (without lost of generality) that $\{P_{n_k}\}_{n_k}$ converges to some $0$-symmetric compact convex set $R\subset\otimes_{i=1}^l\mathbb{R}^{d_i}.$ Indeed, since $conv_\otimes(P_{n_k})\subseteq P_{n_k},$ then $P\subset R$ and so $R$ is a $0$-symmetric convex body. Finally, since $\mathcal{B}_\otimes(d_1,\ldots,d_l)$ is closed in $\mathcal{B}(\otimes_{i=1}^l\mathbb{R}^{d_i})$ 
\cite[Proposition 3.4]{topologytensorialbodies}, we know that $R\in\mathcal{B}_\otimes(d_1,\ldots,d_l).$ This proves that $conv_\otimes^{-1}(Y)$ is compact, as desired.

\

2. Let $W:\mathcal{B}_\otimes(d_1,\ldots,d_l)\times [0,1]\rightarrow\mathcal{B}_\otimes(d_1,\ldots,d_l) $ be defined as
$$
W(P,t)=(1-t)P+tconv_\otimes(P)
$$
We will show that $W$ is the desired $GL_\otimes$-homotopy. Let $P^i,$ $i=1,\dots,l,$ be the convex bodies associated to $P$ in Remark \ref{rem:Q-ifixed}.
Notice that $W$ is well-defined since $P$ and $conv_\otimes(P)$ belong to $\mathcal{B}_{P^{1},\dots,P^{l}}(d_1,\dots,d_l)$ which is convex \cite[Proposition 3.4]{topologytensorialbodies}.
The continuity and equivariance of $W$ follow directly from those of $conv_\otimes.$ Also, $W(P,0)=P$ and $W (P,1)=conv_\otimes(P),$ for every tensorial body $P.$ Furthermore, if $P\in\Pi,$ then $conv_\otimes(P)=P$ and
$$W(P,t)=(1-t)P+tconv_\otimes(P)=(1-t)P+tP=P,$$ for any $t\in[0,1].$ This finishes the proof.

\end{proof}

\begin{theorem}
\label{thm:structurePi}
Let $d_i\geq2,$ $i=1,\dots,l,$ and $p=\frac{d_1(d_1+1)+\cdots+d_l(d_l+1)}{2}.$ The following hold:
\begin{enumerate}
\item $\Pi(d_1,\ldots,d_l)$ is homeomorphic to the orbit space $\mathcal{B}(d_1)\times\cdots\times\mathcal{B}(d_l)/H.$
\item The space of tensorial ellipsoids $\mathcal{E}_\otimes(d_1,\ldots,d_l)$ is homeomorphic to the orbit space $\mathcal{E}(d_1)\times\cdots\times\mathcal{E}(d_l)/H.$ In paticular, the latter is homeomorphic to 
$\mathbb{R}^p.$
\item $\Pi(d_1,\ldots,d_l)$ is homeomorphic to $\mathcal{Q}\times\mathbb{R}^p.$ 
\end{enumerate}
\end{theorem}
\begin{proof}
 1. 
Recall that $\otimes_\pi$ is continuous, see \cite[Proposition 3.3]{topologytensorialbodies}. This, together with (\ref{eq:inverseimage}), shows that
$\otimes_\pi$ induces a continuous bijective map $\tilde{\otimes}_\pi:\mathcal{B}(d_1)\times\cdots\times\mathcal{B}(d_l)/H\rightarrow\Pi$ such that $\otimes_\pi=\tilde{\otimes}_\pi\circ\rho,$ where $\rho:\mathcal{B}(d_1)\times\cdots\times\mathcal{B}(d_l)\rightarrow\mathcal{B}(d_1)\times\cdots\times\mathcal{B}(d_l)/H$ is the orbital projection. 
We will show that $\tilde{\otimes}_\pi$ is indeed a homeomorphism. To this end, let $W\subseteq\mathcal{B}(d_1)\times\cdots\times\mathcal{B}(d_l)$ be a closed subset such that $W=\bigcup\{\rho(P_1,\dots,P_l):(P_1,\dots,P_l)\in W\},$ we will see that $\otimes_\pi(W)$ is closed too. Suppose that $\{P_n\}_{n}\subset\otimes_\pi(W)$ is a sequence converging to some $P\in\Pi,$ and let $P_n^i\subset\mathbb{R}^{d_i},$ $i=1,\dots,l,$ be the $0$-symmetric convex bodies associated to $P_n$ in Remark \ref{rem:Q-ifixed}.  Since  
$P_n=P_{1n}\otimes_\pi\cdots\otimes_\pi P_{ln}$  for some $(P_{1n},\ldots,P_{ln})\in W,$ then 
$P_n\in\mathcal{B}_{P_{1n},\ldots,P_{ln}}\cap\mathcal{B}_{P_n^1,\ldots,P_n^l}$ and, by \cite[Proposition 3.6]{tensorialbodies}, $P_n^i=\lambda_n^iP_{in}$ for some $\lambda_n^i>0$ such that $\lambda_n^1\cdots\lambda_n^l=1.$ So, 
$(P_n^1,\ldots,P_n^l)\in W.$ Using the fact that $P_n^i$ converges to $P^i$ (see \cite[Lemma 3.1 and Proposition 3.7]{topologytensorialbodies}),  we know that $P_n^1\otimes_\pi\cdots\otimes_\pi P_n^l$ approaches to 
$P^1\otimes_\pi\cdots\otimes_\pi P^l$ and the latter belongs to $\otimes_\pi(W).$ Furthermore, since
\begin{align*}
P_n^1\otimes_\pi\cdots\otimes_\pi P_n^l=\lambda_n^1 P_{1n}\otimes_\pi\cdots\otimes_\pi\lambda_n^l P_{ln}
=P_{1n}\otimes_\pi\cdots\otimes_\pi P_{ln}
=P_n,
\end{align*}
then $P=P^1\otimes_\pi\cdots\otimes_\pi P^l\in\otimes_\pi(W).$ This shows that $\otimes_\pi(W)$ is closed, as desired.

\

2. This proof follows along the same ideas of 1. However, we have included it for the sake of completeness. Notice that the restriction $\rho_{|}$ of the orbit map $\rho$ to $\mathcal{E}(d_1)\times\cdots\times\mathcal{E}(d_l)$ is precisely the orbital projection onto 
$\mathcal{E}(d_1)\times\cdots\times\mathcal{E}(d_l)/H.$ In addition, by \cite[Proposition 5]{maiteluisa2}, for any $l$-tuple $E_i\in\mathcal{E}(d_i),$ $i=1,\ldots,l,$ we have that
\begin{align*}
\lambda_1E_1\otimes_2\cdots\otimes_2\lambda_lE_l=(\lambda_1\cdots\lambda_l)E_1\otimes_2\cdots\otimes_2 E_l
=E_1\otimes_2\cdots\otimes_2 E_l,
\end{align*}
whenever each $\lambda_i>0$ and $\lambda_1\cdots\lambda_l=1.$ In consequence, the $H$-orbit $H((E_1,\ldots,E_l))$ is such that $\otimes_2[H((E_1,\ldots,E_l))]=\{E_1\otimes_2\cdots\otimes_2 E_l\}.$ This, along with Proposition \ref{prop:continuity of hilbert}, yield that $\otimes_2$ determines a continuous injective map $\tilde{\otimes}_2:\mathcal{E}(d_1)\times\cdots\times\mathcal{E}(d_l)/H\rightarrow\mathcal{E}_\otimes(d_1,\ldots,d_l)$ such that $\otimes_2=\tilde{\otimes}_2\circ\rho_{|}.$ Moreover, given that in \cite[Corollary 4.3]{tensorialbodies}, it is proved that  
$\mathcal{E}_\otimes$ consists only of Hilbertian tensor products of ellipsoids $E_i\in\mathcal{E}(d_i),$ $i=1,\ldots,l$, then $\tilde{\otimes}_2$ is a surjective map. 

We now turn to prove that $\tilde{\otimes}_2$ is a homeomorphim. To this end, we will consider a closed subset $W\subseteq\mathcal{E}(d_1)\times\cdots\times\mathcal{E}(d_l)$ such that $W=\bigcup\{\rho_{|}(E_1,\ldots,E_l): (E_1,\ldots,E_l)\in W \},$ and we shall prove that $\otimes_2(W)$ is a closed set. Let $\{E_n\}_n\subseteq \otimes_2(W)$ be a sequence converging to $E\in\mathcal{E}_\otimes.$  Clearly $E_n=E_{1n}\otimes_2\cdots\otimes_2 E_{ln},$ for some $(E_{1n},\ldots,E_{ln})\in W.$ Also, by \cite[Proposition 3.6]{tensorialbodies} and the fact that $E_n$ belongs to $\mathcal{B}_{E_n^1,\ldots,E_n^l}\cap\mathcal{B}_{E_{1n},\ldots,E_{ln}},$ it follows that
$E_n^i=\lambda_n^i E_{in},$ with $\lambda_n^i>0$ and $\lambda_n^1\cdots\lambda_n^l=1.$ Thus,  
$(E_n^1,\ldots, E_n^l)\in W$ and, since $E_n^i$ converges to $E^i$ (see \cite[Lemma 3.1 and Proposition 3.7]{topologytensorialbodies}), we have that $(E^1,\ldots, E^l)\in W.$ Indeed, by the continuity of $\otimes_2$ (Proposition \ref{prop:continuity of hilbert}), we also know that $E_n^1\otimes_2\cdots\otimes_2 E_n^l$ approaches to $E^1\otimes_2\cdots\otimes_2 E^l.$ Finally, since
\begin{align*}
E_n^1\otimes_2\cdots\otimes_2 E_n^l=\lambda_n^1 E_{1n}\otimes_2\cdots\otimes_2\lambda_n^l E_{ln}
=E_{1n}\otimes_2\cdots\otimes_2E_{ln}
=E_n,
\end{align*}
then  $E_n^1\otimes_2\cdots\otimes_2 E_n^l$ converges to $E$ and $E=E^1\otimes_2\cdots\otimes_2 E^l\in\otimes_2(W).$ This proves that $\otimes_2(W)$ is closed as required. 
The last part of the result now follows from the fact that $\mathcal{E}_\otimes(d_1,\ldots,d_l)$ is homeomorphic to $\mathbb{R}^p$ \cite[Corollary 4.2]{topologytensorialbodies}. 

\

3. 
Let $L(d_i)\subset\mathcal{B}(d_i)$ be the subspace consisting of $P\in\mathcal{B}(d_i)$ such that $Low(P)=B_2^{d_i}.$ We shall prove that $\mathcal{B}(d_1)\times\cdots\times\mathcal{B}(d_l)$ is homeomorphic to the product $L(d_1)\times\cdots\times L(d_l)\times\left(\mathcal{E}(d_1)\times\cdots\times\mathcal{E}(d_l)/H\right).$ The result then will be a consequence of  2., and the fact that $L(d_i)$ is homeomorphic to $\mathcal{Q}$ \cite[Theorem 1.4]{antonyananerviowest}, and 
$\mathcal{Q}\times\mathcal{Q}$ is homeomorphic to $\mathcal{Q}$ (see \cite[Excercise 2, \textsection 1.1]{vanMillbook}). 
We will require some notation from  \cite{AntonyanNatalia}.
Let $\ell_i:\mathcal{B}(d_i)\rightarrow\mathcal{E}(d_i)$ be the $GL(d_i)$-map sending each $P_i\in\mathcal{B}(d_i)$ to its L\"{o}wner ellipsoid $\ell_i(P_i)=Low(P_i),$ and let $r_i:\mathcal{B}(d_i)\rightarrow L(d_i)$ be the $O(d_i)$-retraction of \cite[Corollary 3.9]{AntonyanNatalia} defined as
$$r_i(P_i)=[\xi_i(\ell_i(P_i))]^{-1}P_i.$$
Here $\xi_i$ is the homeomorphism sending a $0$-symmetric ellipsoid $C\subset\mathbb{R}^{d_i}$ to the unique self-adjoint positive operator $\xi_i(C)\in GL(d_i)$ such that $\xi_i(C)B_2^{d_i}=C,$ $i=1,\ldots,l,$ see  \cite[(3.4)]{AntonyanNatalia}.
We claim that the diagram below commutes and the map $\tilde{\varphi}$ is a homeomorphism:
\begin{center}
\begin{tikzpicture}
  \matrix (m) [matrix of math nodes,row sep=3em,column sep=4em,minimum width=2em]
  {
     \mathcal{B}(d_1)\times\cdots\times\mathcal{B}(d_l) & L(d_1)\times\cdots\times L(d_l)\times\mathcal{E}(d_1)\times\cdots\times\mathcal{E}(d_l) \\
     \mathcal{B}(d_1)\times\cdots\times\mathcal{B}(d_l)/H & L(d_1)\times\cdots\times L(d_l)\times\left(\mathcal{E}(d_1)\times\cdots\times\mathcal{E}(d_l)/H\right) \\};
  \path[-stealth]
    (m-1-1) edge node [left] {$\rho$} (m-2-1)
            edge node [below] {$\varphi$} (m-1-2)
    (m-1-2) edge node [right] {$I_{\prod_i L(d_i)}\times\rho_{|}$} (m-2-2)
    (m-2-1) edge node [below] {$\tilde{\varphi}$} (m-2-2);
\end{tikzpicture}
\end{center}
Here $I_{\prod_i L(d_i)}$ denotes the identity map of $L(d_1)\times\cdots\times L(d_l)$ and  $\varphi$ is defined as 
$$
\varphi(P_1,\ldots,P_l)=\left(r_1(P_1),\ldots,r_l(P_l),\ell_1(P_1),\ldots,\ell_l(P_l)\right),
$$
for $P_i\in\mathcal{B}(d_i),$ $i=1,\ldots,l.$ Observe that $\varphi$ is a continuous map, since both  $r_i$ and $\ell_i$ are continuous for each $i$. Moreover, due to the map sending $P_i\in\mathcal{B}(d_i)$ to the pair $(r_i(P_i),\ell_i(P_i))$ is a $O(d_i)$-homeomorphism between $\mathcal{B}(d_i)$ and $L(d_i)\times\mathcal{E}(d_i) $ (see \cite[Corollary 3.9]{AntonyanNatalia}), then  $\varphi$ must be a $O(d_1)\times\cdots\times O(d_l)$-equivariant homeomorphism. 

On the other hand, if we let $\lambda>0$ and $X\in\mathcal{B}(d_i),$ then $\lambda\xi[\ell_i(X)]$ is a positive self-adjoint operator such that  
\begin{equation*}
\lambda\xi[\ell_i(X)]B_2^{d_i}=\lambda\ell_i(X)=\ell_i(\lambda X).
\end{equation*}
Hence, from the uniqueness of  $\xi[\ell_i(\lambda X)],$ we know that $\xi[\ell_i(\lambda X)]=\lambda\xi[\ell_i(X)]$ and thus 
\begin{equation}
\label{eq:r is homot invariant}
r_i(\lambda X)=X, \text{ for any } \lambda>0.
\end{equation}

Now, if we consider arbitrary $l$-tuples $(P_1,\ldots,P_l)$ and $(X_1,\ldots,X_l)$ in the same $H$-orbit, then there exists $\lambda_i>0,$ $i=1,\ldots,l,$ such that $\lambda_1\cdots\lambda_l=1$ and $X_i=\lambda_i P_i.$ Therefore, by (\ref{eq:r is homot invariant}),
\begin{align*}
\varphi(X_1,\ldots,X_l)=&\left(r_1(P_1),\ldots,r_l(P_l),\ell_1(\lambda_1P_1),\ldots,\ell_l(\lambda_l P_l)\right)\\
=&\left(r_1(P_1),\ldots, r_l(P_l),\lambda_1\ell_1(P_1),\ldots,\lambda_l\ell_l( P_l)\right),
\end{align*}
and so $I_{\prod_i L(d_i)}\times\rho_{|}\left(\varphi(X_1,\ldots,X_l)\right)=I_{\prod_i L(d_i)}\times\rho_{|}\left(\varphi(P_1,\ldots,P_l)\right).$ Similarly, by the injectivity of $\varphi$, if the last equality holds then $(P_1,\ldots,P_l)$ and $(X_1,\ldots,X_l)$ must be in the same $H$-orbit.  
Thus, due to  $I_{\prod_i L(d_i)}\times\rho_{|}$ is indeed a quotient map and $\varphi$ is a homeomorphism, this shows that $\tilde{\varphi}$ is well-defined, bijective and it must be a homeomorphism.
\end{proof}

\section{On the ANR property}
\label{sec:ANRproperty}

The support function $h_P:\mathbb{R}^d\rightarrow[0,\infty)$ of a $0$-symmetric convex body $P\subset\mathbb{R}^{d}$ is defined as $h_P(y)=\text{sup}_{x\in P}\langle x,y\rangle.$ Similarly, for a $0$-symmetric convex body $P\subset\otimes_{i=1}^l\mathbb{R}^{d_i},$ its support function $h_P:\otimes_{i=1}^l\mathbb{R}^{d_i}\rightarrow[0,\infty)$ is calculated with respect to the Euclidean structure $\langle\cdot,\cdot\rangle_H$ of the tensor space. It is well-known that $h_P(\cdot)$ coincides with the Minkowski functional of the polar body. That is $h_P(\cdot)=g_{P^\circ}(\cdot)$ (\cite[Theorem 1.7.6]{Schneider1993}). 

\begin{proposition}
\label{prop:notconvex}
The space of tensorial bodies $\mathcal{B}_\otimes(d_1,\ldots,d_l),$  $d_i\geq 2,$ $i=1,\ldots,l$ and $l\geq2,$  is not closed with respect to the Minkowski sum. 
\end{proposition}

\begin{proof}
Let $B_p^{d_i}$ denote the closed unit ball of the classical norm $\|\cdot\|_p$ on $\mathbb{R}^{d_i},$ $i=1,\dots,l,$ for $p=1,\infty$.  We shall show that the Minkowski sum of the projective tensor products  $P=\frac{1}{2}B_{2}^{d_1}\otimes_{\pi}\cdots\otimes_{\pi}B_{2}^{d_l}$ and  
$R=\frac{1}{2}B_{1}^{d_1}\otimes_{\pi}\cdots\otimes_{\pi}B_{1}^{d_l}$ is not a tensorial body in $\otimes_{i=1}^l\mathbb{R}^{d_i}$. 
Since the support function is Minkowski additive, we know that $h_{P+Q}(\cdot)=h_{P}(\cdot)+h_{Q}(\cdot)$  and so $g_{(P+R)^\circ}(\cdot)=g_{P^\circ}(\cdot)+ g_{R^\circ}(\cdot).$ Hence, from (\ref{eq:rc in gauges2}) and the fact that  $(B_1^{d_i})^\circ=B_\infty^{d_i}$, we get
$$
g_{(P+R)^\circ}(x^1\otimes\cdots\otimes x^l)=\frac{\|x^1\|_2\cdots\|x^l\|_2+\|x^1\|_\infty\cdots\|x^l\|_\infty}{2},
$$
for any $x^i\in\mathbb{R}^{d_i},$ $i=1,\dots,l$. Thus, if we fix the vectors $a^i=(1,\ldots,1)\in\mathbb{R}^{d_i}$ and $b^i=(1,0,\ldots,0)\in\mathbb{R}^{d_i}$ with $i=1,\ldots,l-1,$ then, for every $x^l\in\mathbb{R}^{d_l}$ we have
$$ 
g_{(P+R)^\circ}(a^1\otimes\cdots\otimes a^{l-1}\otimes x^l)=\frac{\sqrt{d_1\cdots d_{l-1}}\|x^l\|_2+\|x^l\|_\infty}{2},
$$
and
$$
g_{(P+R)^\circ}(b^1\otimes\cdots\otimes b^{l-1}\otimes x^l)=\frac{\|x^l\|_2+\|x^l\|_\infty}{2}.
$$
Therefore, the value of $g_{(P+R)^\circ}(a^1\otimes\cdots\otimes a^{l-1}\otimes x^l)/g_{(P+R)^\circ}(b^1\otimes\cdots\otimes b^{l-1}\otimes x^l)$ depends on $x^l\neq 0.$ In consequence, there do not exist $0$-symmetric convex bodies such that (\ref{eq:rc in gauges}) holds for $(P+R)^\circ.$ In particular, it is not a tensorial body. Since $P+R$ belongs to the class of tensorial bodies if and only if $(P+R)^\circ$ does it too (\cite[Proposition 3.5]{tensorialbodies}), the proof is finished.
\end{proof}

As we just showed in the proof of the previous proposition, the Minkowski sum of an arbitrary pair of tensorial bodies is not necessarily a tensorial body. Nevertheless, taking into account the multilinear nature of tensors, it turns out that  in many cases the Minkowski sum does remain in the class.

\begin{proposition}
\label{prop:multilineality}
Let $\lambda>0$ be a real number and let $P_i\in\mathcal{B}(d_{i}),$ $i=1,\dots,l,$ be an arbitrary $l$-tuple. Then, for every $j=1,\dots,l$ and any 
$R_j\in\mathcal{B}(d_{j}),$ if  $P\in\mathcal{B}_{P_1,\dots,P_j,\dots,P_l}(d_1,\dots,d_l)$ and $R\in\mathcal{B}_{P_1,\dots,P_{j-1},R_j,P_{j+1},\dots,P_l}(d_1,\dots,d_l),$ then $P+\lambda R$ is a tensorial body. Indeed, $P+\lambda R\in\mathcal{B}_{P_1,\dots,P_{j-1},P_j+\lambda R_j,P_{j+1},\dots,P_l}(d_1,\dots,d_l).$
\end{proposition}
\begin{proof}
For simplicity of the notation, we prove the result for $j=1.$ Notice that if $x^1\in P_1,$ $y^1\in R_1$ and $x^i\in P_i,$ for $i=2,\dots,l,$ then 
\begin{align*}
(x^1+\lambda y^1)\otimes x^2\otimes\cdots\otimes x^l&=x^1\otimes x^2\otimes\cdots\otimes x^l+\lambda y^1\otimes x^2\otimes\cdots\otimes x^l\\
&\in P_1\otimes_\pi P_2\otimes_\pi\cdots\otimes_\pi  P_l+\lambda R_1\otimes_\pi P_2\otimes_\pi\cdots\otimes_\pi  P_l\\
&\subseteq P+\lambda R.
\end{align*}
Thus, by the definition of $\otimes_\pi,$ $(P_1+\lambda R_1)\otimes_\pi P_2\otimes_\pi\cdots\otimes_\pi  P_l\subseteq P+\lambda R.$
To prove the inclusion $P+\lambda R\subseteq(P_1+\lambda R_1)\otimes_\epsilon P_2\otimes_\epsilon\cdots\otimes_\epsilon  P_l,$ observe that 
$h_{P+\lambda R}(\cdot)=h_{P}(\cdot)+h_{\lambda R}(\cdot)$  and so $g_{(P+\lambda R)^\circ}(\cdot)=g_{P^\circ}(\cdot)+ g_{(\lambda R)^\circ}(\cdot).$ Hence, by (\ref{eq:rc in gauges2}), for every $x^i\in\mathbb{R}^{d_i},$ $i=1,\dots,l,$ we also have
\begin{align*}
g_{(P+\lambda R)^\circ}(x^1\otimes\cdots\otimes x^l)&=g_{P_1^\circ}(x^1)g_{P_2^\circ}(x^2)\cdots g_{P_l^\circ}(x^l)+g_{(\lambda R_1)^\circ}(x^1)g_{P_2^\circ}(x^2)\cdots g_{P_l^\circ}(x^l)\\
&=(g_{P_1^\circ}(x^1)+g_{(\lambda R_1)^\circ}(x^1))g_{P_2^\circ}(x^2)\cdots g_{P_l^\circ}(x^l).
\end{align*}
This, together with the definition of $\otimes_\epsilon,$ yields to $P+\lambda R\subseteq(P_1+\lambda R_1)\otimes_\epsilon P_2\otimes_\epsilon\cdots\otimes_\epsilon  P_l.$ Finally, from (\ref{eq:definition tensorbody}), it follows that $P+\lambda R$ is a tensorial body w.r.t. $P_1+\lambda R_1,P_{2},\dots,P_l$. This finishes the proof. 
\end{proof}

As an application of Proposition \ref{prop:multilineality}, notice that it is always possible to define a polygonal path joining any pair $P,R$ of tensorial bodies in $\mathbb{R}^{d_1}\otimes\mathbb{R}^{d_2}.$ To this end, let us assume that $P\in\mathcal{B}_{P_1,P_2}(d_1,d_2)$ and $R\in\mathcal{B}_{R_1,R_2}(d_1,d_2),$ for some $P_i,R_i\in\mathcal{B}(d_i),$ $i=1,2,$ and let $K$ be an arbitrary tensorial body w.r.t. $P_1,R_2.$ Then, the map $\alpha:[0,1]\rightarrow\mathcal{B}_\otimes(d_1,d_2)$ defined as,
\begin{equation*}
\alpha(t)=\begin{cases}
(1-2t)P+(2t)K; \text{ for } 0\leq t\leq\frac{1}{2},\\
(2-2t)K+(2t-1)R; \text{ for } \frac{1}{2}\leq t\leq1,
\end{cases}
\end{equation*}
 is a polygonal path joining $P$ and $R.$ It is not difficult to see that this process can be extended to 
$\mathcal{B}_\otimes(d_1,\dots,d_l)$ with $l>2.$

\begin{theorem}
\label{thm:ANRproperty}
$\mathcal{B}_\otimes(d_1,\dots,d_l)$, $d_i\geq2,$ $i=1,\dots,l,$ is an AR. In particular, the subspace $\mathcal{L}_\otimes(d_1,\dots,d_l)$ is also an AR.
\end{theorem}

\begin{proof}
We shall exhibit a retraction from the space  $\mathcal{B}(\otimes_{i=1}^{l}\mathbb{R}^{d_i})$ onto $\mathcal{B}_\otimes(d_1,\dots,d_l).$ The result then will follow from the fact that  $\mathcal{B}(\otimes_{i=1}^{l}\mathbb{R}^{d_i})$ is an AR, and that any retract of an AR is also an AR \cite[\textsection1.5]{vanMillbook}.

First, for every $P\in\mathcal{B}(\otimes_{i=1}^{l}\mathbb{R}^{d_i})$, let us define the $0$-symmetric convex bodies $P^i\subset\mathbb{R}^{d_i},$ $i=1,\dots,l,$ associated to $P$ as
$$
P^i:=\left\{x^i\in\mathbb{R}^{d_i}:\frac{e^{d_1}_{1}\otimes\cdots\otimes e^{d_{i-1}}_{1}\otimes x^i\otimes e^{d_{i+1}}_{1}\otimes\cdots\otimes e^{d_l}_{1}}{g_{P}(e^{d_1}_{1}\otimes e^{d_2}_{1}\otimes\cdots\otimes e^{d_l}_{1})}\in P\right\},
$$
when $i\leq l-1,$ and
$
P^l:=\left\{x^l\in\mathbb{R}^{d_l}:e^{d_1}_{1}\otimes\cdots\otimes e^{d_{l-1}}_{1}\otimes x^l\in P\right\}.
$
It is not difficult to check that the Minkowski functional associated to each $P^i$ is such that
$$
g_{P^i}(x^i)=\frac{g_P(e^{d_1}_{1}\otimes\cdots\otimes e^{d_{i-1}}_{1}\otimes x^i\otimes e^{d_{i+1}}_{1}\otimes\cdots\otimes e^{d_l}_{1})}{g_{P}(e^{d_1}_{1}\otimes e^{d_2}_{1}\otimes\cdots\otimes e^{d_l}_{1})},
$$
$i=1,\dots,l-1,$ and $g_{P^l}(x^l)=g_P(e^{d_1}_{1}\otimes\cdots\otimes e^{d_{l-1}}_{1}\otimes x^l).$ Indeed, for every $i,$ the map from $\mathcal{B}(\otimes_{i=1}^{l}\mathbb{R}^{d_i})$ to $\mathcal{B}(d_i)$ sending $P\mapsto P^i$ must be continuous. To prove this, recall that for every sequence $\{P_n\}_n$ and $P$ in $\mathcal{B}(\otimes_{i=1}^{l}\mathbb{R}^{d_i}),$ if $P_n$ converges to $P,$ then the sequence of the Minkowski functionals $\left\{g_{P_n}\right\}_n$ converges uniformly to $g_P$ on compact sets of $\otimes_{i=1}^{l}\mathbb{R}^{d_i}$ (\cite[Lemma 3.1]{topologytensorialbodies}). Hence, for every $x^i\in\mathbb{R}^{d_i},$ we have that $g_{P_n^i}(x^i)$ approaches to $g_{P^i}(x^i)$. So, the sequence $\left\{g_{P_n^i}\right\}_n$ converges pointwise to $g_{P^i}$ for each $i.$ Since the pointwise convergence of Minkowski functionals implies the uniform convergence on compact sets, $\left\{g_{P_n^i}\right\}_n$ converges  uniformly to $g_{P^i}$ on compact sets of $\mathbb{R}^{d_i},$ $i=1,\ldots,l.$ In particular, each $P_n^i$ converges to $P^i$ with respect to the Hausdorff metric \cite[Lemma 3.1]{topologytensorialbodies}. This proves that  the map in question is continuous.

Next, we claim that the map $\eta:\mathcal{B}(\otimes_{i=1}^{l}\mathbb{R}^{d_i})\rightarrow\mathcal{B}_\otimes(d_1,\dots,d_l)$ defined as,
$$
\eta(P)=\text{conv}(P\cup P^1\otimes_\pi P^2\otimes_\pi\cdots\otimes_\pi P^l)\cap P^1\otimes_\epsilon P^2\otimes_\epsilon\cdots\otimes_\epsilon P^l,
$$
 is the desired retraction. To prove this, notice that for every $0$-symmetric convex body $P\subset\otimes_{i=1}^{l}\mathbb{R}^{d_i},$ $\eta(P)$ is a $0$-symmetric convex body such that
$$
P^1\otimes_\pi P^2\otimes_\pi\cdots\otimes_\pi P^l\subseteq \eta(P)\subseteq P^1\otimes_\epsilon P^2\otimes_\epsilon\cdots\otimes_\epsilon P^l,
$$
and thus, by (\ref{eq:definition tensorbody}), $\eta(P)$ is a tensorial body. Furthermore,  from the previous argument in combination with the continuity of both the projective $\otimes_\pi$ and the injective $\otimes_\epsilon$ tensor products \cite[Proposition 3.3]{topologytensorialbodies}, and the fact that all the operators involved are continuous  in the class of $0$-symmetric convex bodies (see \cite[Section 1.8]{Schneider1993}), we know that $\eta$ is well-defined and continuous.

Finally, if $P$ is a tensorial body in $\otimes_{i=1}^{l}\mathbb{R}^{d_i},$ then, by Remark \ref{rem:Q-ifixed}, we know that $P^1\otimes_\pi P^2\otimes_\pi\cdots\otimes_\pi P^l\subseteq P\subseteq P^1\otimes_\epsilon P^2\otimes_\epsilon\cdots\otimes_\epsilon P^l.$ Therefore $\eta(P)=P,$ and this shows that $\eta$ is a retraction onto the space of tensorial bodies.
The last part of the statement now follows from the fact that $\mathcal{L}_\otimes(d_1,\dots,d_l)$ is a retract of the space of tensorial bodies \cite[Corollary 4.10]{topologytensorialbodies}, and thus it is also an AR.
\end{proof}

As it was described in Section \ref{sec:tensorialbodiesprev}, for a fixed $l$-tuple $(P_1,\dots,P_l)\in\mathcal{B}(d_1)\times\cdots\times\mathcal{B}(d_l),$ the hyperspace $\mathcal{B}_{P_1,\dots,P_l}(d_1,\dots,d_l)$ consists of the tensorial bodies $P\subset\otimes_{i=1}^{l}\mathbb{R}^{d_i}$ such that
$$
P_{1}\otimes_{\pi}\cdots\otimes_{\pi}P_{l}\subseteq P\subseteq P_{1}\otimes_{\epsilon}\cdots\otimes_{\epsilon}P_{l}.
$$
In \cite[Proposition 3.1]{tensorialbodies}, it is proved that this hyperspace coincides with the set of closed unit balls of reasonable crossnorms on the space $\otimes_{i=1}^{l}M_i,$ with $M_i=(\mathbb{R}^{d_i},g_{P_i})$. In this case, the map sending $P$ to its Minkowski functional $g_P$ gives us a bijective correspondence between $\mathcal{B}_{P_1,\dots,P_l}(d_1,\dots,d_l)$ and the set of reasonable crossnorms on $\otimes_{i=1}^{l}M_i.$ Since, each tensorial body $P\subset\otimes_{i=1}^{l}\mathbb{R}^{d_i}$ belongs to $\mathcal{B}_{P_1,\dots,P_l}(d_1,\dots,d_l),$ for some $P_i\in\mathcal{B}(d_i),$ $i=1,\dots,l,$  the following relation holds:
$$
\mathcal{B}_{\otimes}(d_1,\dots.,d_l)=\bigcup_{(P_1,\dots,P_l)\in\mathcal{B}(d_1)\times\cdots\times\mathcal{B}(d_l)}\mathcal{B}_{P_1,\dots,P_l}(d_1,\dots,d_l).
$$
 It is noteworthy that, in contrast with $\mathcal{B}_{\otimes}(d_1,\dots.,d_l)$, each $\mathcal{B}_{P_1,\dots,P_l}(d_1,\dots,d_l)$ is a compact and convex subspace of  $\mathcal{B}(\otimes_{i=1}^{l}\mathbb{R}^{d_i})$  \cite[Proposition 3.4]{topologytensorialbodies}. These properties allow us to determine its topological structure.

\begin{proposition}
\label{prop:globosHilbertcube}
Let $d_i\geq 2$, $i=1,\ldots,l.$  Then, for every tuple of $0$-symmetric convex bodies  $P_i\subset\mathbb{R}^{d_i}$, the hyperspace $\mathcal{B}_{P_1,\ldots,P_l}(d_1,\ldots,d_l)$ is homeomorphic to $\mathcal{Q}.$ 
\end{proposition}
\begin{proof}
The proof is a direct consequence of Lemma \ref{lem:infinitedimension}, and the fact that  $P_1\otimes_\pi\cdots\otimes_\pi P_l\subsetneqq P_1\otimes_\epsilon\cdots\otimes_\epsilon P_l$ for every $l$-tuple $P_i\in\mathcal{B}(d_i),$ $i=1,\dots,l,$ \cite[Theorem 2]{xorgames}.
\end{proof}

\section{Topology of the slice $\mathcal{L}_\otimes(d_1,\dots,d_l)$}
\label{sec:mainresults}

In this section, we achieve our main goal. Namely, we show that $\mathcal{L}_{\otimes}(d_1,\ldots,d_l)$ is homeomorphic to the Hilbert cube $\mathcal{Q}$ (Theorem \ref{thm:maintheorem}) and, in consequence, the space of tensorial bodies $\mathcal{B}_{\otimes}(d_1,\ldots,d_l)$ is homeomorphic to $\mathcal{Q}\times\mathbb{R}^p$ (Corollary \ref{cor:mainresultcor}). 

Recall that  $\mathcal{L}_{\otimes}(d_1,\ldots,d_l)$ consists of the tensorial bodies $P\subset\otimes_{i=1}^l\mathbb{R}^{d_i}$ for which $\ell_{\otimes}(P)=B_{2}^{d_{1},\ldots,d_{l}}.$ It is known that $\mathcal{L}_{\otimes}(d_1,\ldots,d_l)$ is a compact $O_\otimes$-global slice for the $GL_\otimes$-space $\mathcal{B}_{\otimes}(d_1,\ldots,d_l)$ (see Theorem \ref{thm:resumenthm} 
and \cite[Sections 4.1 and 4.2]{topologytensorialbodies}).
For simplicity, we denote by $\Pi\cap\mathcal{L}_{\otimes}$ the subset $\Pi(d_1,\ldots,d_l)\cap\mathcal{L}_{\otimes}(d_1,\ldots,d_l)$ consisting of the  tensor products $P_1\otimes_\pi\cdots\otimes_\pi P_l,$ $P_i\in\mathcal{B}(d_i),$ $i=1,\dots,l,$ such that $\ell_{\otimes}(P_1\otimes_\pi\cdots\otimes_\pi P_l)=B_{2}^{d_{1},\ldots,d_{l}}.$ 
We follow the notation of \cite[Remark 1]{banacamazurcompactum} (c.f. \cite[Section 3.2]{AntonyanNatalia}), where the class of $0$-symmetric convex bodies $P\subset\mathbb{R}^d$ such that $Low(P)=B_2^d$ is denoted by $L(d).$  

\begin{proposition}
\label{prop:Ltensorpoyectivecontraible}
The following hold:
\begin{enumerate}
\item $\Pi\cap\mathcal{L}_{\otimes}$ is an $O_\otimes$-strong deformation retract of $\mathcal{L}_{\otimes}(d_1,\ldots,d_l).$

\item 
The map $\otimes_\pi$ restricted to $L(d_1)\times\cdots\times L(d_l)$ is a homeomorphism onto $\Pi\cap\mathcal{L}_{\otimes}.$ In particular, $\Pi\cap\mathcal{L}_{\otimes}$ is homeomorphic to $\mathcal{Q}$.

\item $\Pi\cap\mathcal{L}_{\otimes}$ is  $O_\otimes$-strictly contractible to the $O_\otimes$-fixed point $B_2^{d_1}\otimes_\pi\cdots\otimes_\pi B_2^{d_l}.$ 
\end{enumerate}
\end{proposition}
\begin{proof}
1. Let us denote by $conv_{\otimes|\mathcal{L}_{\otimes}}$ the restriction of the $GL_\otimes$-map $conv_{\otimes}$ to   $\mathcal{L}_{\otimes}(d_1,\ldots,d_l).$ Below we prove that $conv_{\otimes|\mathcal{L}_{\otimes}}$ is an $O_\otimes$-strong deformation retraction.
To this end, observe that $conv_{\otimes|\mathcal{L}_{\otimes}}$ maps $\mathcal{L}_{\otimes}$ onto $\Pi\cap\mathcal{L}_{\otimes},$ due to the map $conv_{\otimes}$ being a retraction of $\mathcal{B}_\otimes$ onto $\Pi$ \cite[Proposition 4.4]{topologytensorialbodies} and
$
\ell_\otimes(conv_{\otimes}(P))=\ell_\otimes(P)=B_{2}^{d_{1},\ldots,d_{l}},
$
for every $P\in\mathcal{L}_{\otimes}(d_1,\ldots,d_l).$

 Now let us define $H:\mathcal{L}_{\otimes}(d_1,\ldots,d_l)\times[0,1]\rightarrow\mathcal{L}_{\otimes}(d_1,\ldots,d_l)$ as 
$$H(P,t)=(1-t)P+tconv_\otimes(P).$$
We will prove that $H$  is the desired $O_\otimes$-homotopy. Notice that $H$ is well-defined, since $conv_\otimes(P)=P^1\otimes_\pi\cdots\otimes_\pi P^1$, where $P^1,\dots,P^l$ are the convex bodies associated to $P$ in Remark \ref{rem:Q-ifixed}, see \cite[p. 12]{topologytensorialbodies}.
In consequence, both $P$ and $conv_\otimes(P)$ belong to $\mathcal{B}_{P^1,\dots,P^l}(d_1,\dots,d_1),$ which is convex \cite[Proposition 3.4]{topologytensorialbodies}. Thus, $H(P,t)$ is a tensorial body such that
$\ell_\otimes((1-t)P+tconv_\otimes(P))=\ell_\otimes(P)=B_2^{d_1,\ldots,d_l}.$ The $O_\otimes$-equivariance and continuity of $H$ follow directly from those of the map $conv_\otimes$ and the properties of the Minkowski sum.
Clearly, $H(P,0)=P$ and $H(P,1)=conv_\otimes(P)$ for every $P\in\mathcal{L}_{\otimes}(d_1,\dots,d_l).$ Furthermore, if $P\in\Pi\cap\mathcal{L}_{\otimes},$ then $conv_\otimes(P)=P$ and $H(P,t)=P$ for any $t\in[0,1].$ Thus, $H$ is an $O_\otimes$-strong homotopy as required.

\

2. Let us first prove that $\otimes_\pi$ restricted to $L(d_1)\times\cdots\times L(d_l)$ is a homeomorphism onto its image $\otimes_\pi(L(d_1),\ldots,L(d_l)).$ 
To this end, notice that  $\otimes_{\pi|L(d_1)\times\cdots\times L(d_l)}$ is surjective and, since $\otimes_\pi$ is continuous and $L(d_1)\times\cdots\times L(d_l)$ is compact (\cite[Proposition 3.4]{AntonyanNatalia}), $\otimes_{\pi|L(d_1)\times\cdots\times L(d_l)}$ must be a closed map. To prove that it is injective, suppose that $R_1\otimes_\pi\cdots\otimes_\pi R_l=P_1\otimes_\pi\cdots\otimes_\pi P_l,$ with $R_i,P_i\in L(d_i),$ then, by \cite[Proposition 3.6]{tensorialbodies}, there exist $\lambda_i>0,$ $i=1,\ldots,l,$ such that $\lambda_1\cdots\lambda_l=1$ and  $P_i=\lambda_{i}R_i.$ However, each pair $R_i,P_i\in L(d_i),$ hence $B_2^{d_i}=Low(P_i)=\lambda_i Low(R_i)=\lambda_i B_2^{d_i},$ and so  $\lambda_i=1.$  This shows that $P_i=R_i$ and therefore  $\otimes_{\pi|L(d_1)\times\cdots\times L(d_l)}$ is a homeomorphism.

Now we will prove that $\otimes_\pi(L(d_1),\ldots,L(d_l))=\Pi\cap\mathcal{L}_{\otimes}.$ Let $R_i\in L(d_i)$ for $i=1,\ldots,l.$ Since the L\"{o}wner ellipsoid of $R_1\otimes_\pi\cdots\otimes_\pi R_l$ is the Hilbertian tensor product $Low(R_1)\otimes_2\cdots\otimes_2Low(R_l)$  (see \cite[Lemma 1]{Aubrun2006}), we have
\begin{equation*}
    \ell_\otimes(R_1\otimes_\pi\cdots\otimes_\pi R_l)=Low(R_1)\otimes_2\cdots\otimes_2Low(R_l)
    =B_2^{d_1}\otimes_2\cdots\otimes_2B_2^{d_l}
    =B_2^{d_1,\ldots,d_l}.
\end{equation*}
Hence $\otimes_\pi(L(d_1),\ldots,L(d_l))$ is contained in $\Pi\cap\mathcal{L}_{\otimes}.$ 
On the other hand, if $R\in\Pi\cap\mathcal{L}_{\otimes},$ then $R=conv_\otimes(R)=R^1\otimes_\pi\cdots\otimes_\pi R^l$ (see \cite[p. 12]{topologytensorialbodies}) and 
$$\ell_\otimes(R)=Low(R^1)\otimes_2\cdots\otimes_2Low(R^l)
    =B_2^{d_1}\otimes_2\cdots\otimes_2B_2^{d_l}.$$
Thus, by \cite[Proposition 3.6]{tensorialbodies}, there exist $\lambda_i>0,$ $i=1,\ldots,l,$ such that $\lambda_1\cdots\lambda_l=1$ and $Low(R^i)=\lambda_{i}B_2^{d_i}.$ Therefore,  $P_i:=\lambda_{i}^{-1}R^i$ belongs to $L(d_i)$ and, from \cite[Proposition 5]{maiteluisa2}, we have
\begin{align*}
R=R^1\otimes_\pi\cdots\otimes_\pi R^l=\lambda_{1}P_1\otimes_\pi\cdots\otimes_\pi\lambda_{l}P_l
=(\lambda_{1}\cdots\lambda_{l})P_1\otimes_\pi\cdots\otimes_\pi P_l
=P_1\otimes_\pi\cdots\otimes_\pi P_l.
\end{align*}
So, $R\in\otimes_\pi(L(d_1),\ldots,L(d_l)).$ This proves the reverse inclusion.
The result now follows from two facts: Each $L(d_i),$ $i=1,\ldots,l,$ is homemorphic to the Hilbert cube
 $\mathcal{Q}$ \cite{banacamazurcompactum}, and  $\mathcal{Q}^k,$ $k\in\mathbb{N},$ is homeomorphic to $\mathcal{Q}$. 
(see \cite[Excercise 2, \textsection 1.1]{vanMillbook}).

\medskip

3. Let us denote by $\rho_i$, $i=1,\ldots,l,$  the canonical projection of $L(d_1)\times\cdots\times L(d_l)$ onto $L(d_i)$. Also, observe as a consequence of 1., the maps defined by $f_i(P)=\rho_i\circ(\otimes_{\pi|L(d_1)\times\cdots\times L(d_l)})^{-1}(P)$,  $P\in\Pi\cap\mathcal{L}_\otimes,$  are continuous on $\Pi\cap\mathcal{L}_\otimes$ and 
\begin{equation}
\label{eq:identity}
f_1(P)\otimes_\pi\cdots\otimes_\pi f_l(P)=P.
\end{equation}
We will prove that the map $F:\Pi\cap\mathcal{L}_{\otimes}\times[0,1]\rightarrow\Pi\cap\mathcal{L}_{\otimes},$ defined as
$$
F(P,t)=((1-t)f_1(P)+tB_2^{d_1})\otimes_\pi\cdots\otimes_\pi((1-t)f_l(P)+tB_2^{d_l}),
$$
is the desired $O_\otimes$-strict homotopy. 
Since  $Low((1-t)f_i(P)+tB_2^{d_i})=B_2^{d_i}$ for each $i,$ then $F(P,t)\in\Pi\cap\mathcal{L}_\otimes$ and $F$ is well-defined. Its continuity is straightforward from the fact  $\otimes_\pi$  and $f_i$ are continuous. Clearly $F(P,0)=f_1(P)\otimes_\pi\cdots\otimes_\pi f_l(P)=P$ and $F(P,1)=B_2^{d_1}\otimes_\pi\cdots\otimes_\pi B_2^{d_l},$ for every $P$.

To prove that $F_t$ is an $O_\otimes$-map for any $t\in[0,1]$, let $U$ be an orthogonal map in $O_\otimes(d_1,\dots,d_l)$ and $P\in\Pi\cap\mathcal{L}_{\otimes}.$ By  \cite[Lemma A.1]{topologytensorialbodies},  $U=U_1\otimes\cdots\otimes U_l\circ U_\sigma,$ for some $U_i\in O(d_i),$ $i=1,\dots,l,$ and $U_\sigma\in\mathcal{P}.$ Hence, by the identity (\ref{eq:identity}), $UP=U(f_1(P)\otimes_\pi\cdots\otimes_\pi f_l(P)),$ for every $P\in \Pi\cap\mathcal{L}_\otimes$ and
\begin{align*}
F_t(UP)=&F_t(U(f_1(P)\otimes_\pi\cdots\otimes_\pi f_l(P)))\\
=&F_t(U_1f_{\sigma(1)}(P)\otimes_\pi\cdots\otimes_\pi U_l f_{\sigma(l)}(P)) \\
=&((1-t)U_1f_{\sigma(1)}(P)+tB_2^{d_1})\otimes_\pi\cdots\otimes_\pi((1-t)U_lf_{\sigma(l)}(P)+tB_2^{d_l})\\
=&U_1((1-t)f_{\sigma(1)}(P)+tB_2^{d_1})\otimes_\pi\cdots\otimes_\pi U_l((1-t)f_{\sigma(l)}(P)+tB_2^{d_l})\\
\overset{*}{=}&U_1\otimes\cdots\otimes U_l U_\sigma ((1-t)f_{1}(P)+tB_2^{d_1})\otimes_\pi\cdots\otimes_\pi((1-t)f_{l}(P)+t B_2^{d_1})\\
=&UF_t(P).
\end{align*}
(*) Follows from the identity (\ref{eq:gltensor image of proj and inj tp}).
Finally, for $t\in(0,1),$ $F(P,t)=B_2^{d_1}\otimes_\pi\cdots\otimes_\pi B_2^{d_l}$ only if there is  $\lambda_i>0,$ $i=1,\ldots,l,$ such that $\lambda_1\cdots\lambda_l=1$ and $(1-t)f_i(P)+tB_2^{d_i}=\lambda_iB_2^{d_i}$ \cite[Proposition 3.6]{tensorialbodies}.  Then, $Low((1-t)f_i(P)+tB_2^{d_i})=Low(\lambda_iB_2^{d_i})$ and $B_2^{d_i}=\lambda_iB_2^{d_i}.$ Thus each $\lambda_i=1$ and, due to $t\in(0,1),$ $f_i(P)=B_2^{d_i}$ for $i=1,\dots,l.$ Therefore, $P=B_2^{d_1}\otimes_\pi\cdots\otimes_\pi B_2^{d_l}.$ This proves that $F$ is a strict $O_\otimes$-homotopy.
\end{proof}

In Proposition \ref{prop:Ltensorpoyectivecontraible} above, we have showed that $\otimes_{\pi|L(d_1)\times\cdots\times L(d_l)}$ is a homeomorphism onto $\Pi\cap\mathcal{L}_{\otimes}.$ In this way, since for each $P\in\mathcal{L}_{\otimes}(d_1,\ldots,d_l),$ $conv_\otimes(P)\in\Pi\cap\mathcal{L}_{\otimes},$  we know that there is only one $l$-tuple $(\bar{P}_1,\ldots,\bar{P}_l)$ in $L(d_1)\times\cdots\times L(d_l)$ such that $conv_\otimes(P)=\bar{P}_1\otimes_\pi\cdots\otimes_\pi\bar{P}_l.$ 

A map $f:A\rightarrow A,$ defined on a metric space $(A,d_A),$ will be called $\varepsilon$-close to the identity of $A$ if and only if $d_A(f(x),x)<\varepsilon,$ for every $x\in A$. 
\begin{lemma}
\label{lem:metalema}
For each $i=1,\ldots,l,$ let $\mathcal{A}_i$ be a non-empty subset of $L(d_i),$ and denote by $\mathcal{A}$  the set 
$(\text{conv}_{\otimes|\mathcal{L}_{\otimes}})^{-1}(\otimes_\pi(\mathcal{A}_1,\ldots,\mathcal{A}_l))\subseteq\mathcal{L}_\otimes(d_1,\ldots,d_l)$. For every $\rho\in(0,\infty),$ let $\mathcal{F}_\rho$  be the family of all $l$-tuples $(f^1,\ldots,f^l)$ of continuous maps 
$f^i:\mathcal{A}_i\rightarrow\mathcal{A}_i$ such that each $f_i$ is $\rho$-close to the identity of  $\mathcal{A}_i$.
Then for every $\varepsilon>0,$ there exists $\delta>0$ such that for each $(f^1,\ldots,f^l)\in\mathcal{F}_\delta,$  the map $\tilde{f}_\varepsilon:\mathcal{A}\rightarrow\mathcal{A}$  defined as,
\begin{equation}
\label{eq:liftingmap}
\tilde{f}_\varepsilon(P)=\text{conv}\left(P\cup f^1(\bar{P}_1)\otimes_\pi\cdots\otimes_\pi f^l(\bar{P}_l)\right)\cap f^1(\bar{P}_1)\otimes_\epsilon\cdots\otimes_\epsilon f^l(\bar{P}_l),
\end{equation}
is a continuous map $\varepsilon$-close to the identity of $\mathcal{A}.$
\end{lemma}

\begin{proof}
To shorten notation, let us simply write $\mathcal{L}_{\otimes}.$ Clearly, $\mathcal{A}$ is a  non-empty subset of $\mathcal{L}_{\otimes}.$ 
Since $\mathcal{L}_\otimes$ is compact \cite[Proposition 4.6]{topologytensorialbodies}, the operations convex hull, union and intersection are uniformly continuous on it  \cite[Section 1.8]{Schneider1993}. For this reason, given $\varepsilon>0,$ there exists $\eta>0$ (only depending on $\varepsilon$) such that if $X,X',Y,Y',Z$ and $Z'$ belong to $\mathcal{L}_\otimes$ and satisfy $\delta^H(X,X')$, $\delta^H(Y,Y')$, $\delta^H(Z,Z')<\eta,$ then
\begin{equation}
\label{eq:auxmap}
\delta^H\left(\text{conv}(X\cup Y)\cap Z,\text{conv}(X'\cup Y')\cap Z'\right)<\varepsilon.
\end{equation} 
On the other hand, from the compactness of each $L(d_i)$ \cite[Proposition 3.4]{AntonyanNatalia}, it follows that the maps $\otimes_\pi$ and $\otimes_\epsilon$ are uniformly continuous on $L(d_1)\times\cdots\times L(d_l)$ \cite[Proposition 3.3]{topologytensorialbodies}. In consequence, for the above mentioned $\eta,$ there exists $\delta>0$  such that if
$\delta^H(P_i,R_i)<\delta,$ $i=1,\ldots,l,$  then
\begin{equation}
\label{eq:ineqtensorproducts}
\delta^H(P_1\otimes_\alpha\cdots\otimes_\alpha P_l,R_1\otimes_\alpha\cdots\otimes_\alpha R_l)<\eta,\quad\alpha=\pi,\epsilon,
\end{equation} 
 for every pair of tuples  $(P_1,\ldots,P_l)$ and $(R_1,\ldots,R_l)$ in 
$L(d_1)\times\cdots\times L(d_l).$ Observe that $\delta$ only depends of $\varepsilon.$ 
Next, let $(f^1,\ldots,f^l)$ be an $l$-tuple in $\mathcal{F}_\delta$. Then for every $P\in\mathcal{A},$ 
$\delta^H(f^i(\bar{P}_i),\bar{P}_i)<\delta,$ $i=1,\dots,l.$ Hence, by (\ref{eq:ineqtensorproducts}),  
$\delta^H(f^1(\bar{P}_1)\otimes_\alpha\cdots\otimes_\alpha f^l(\bar{P}_l),\bar{P}_1\otimes_\alpha\cdots\otimes_\alpha \bar{P}_l)<\eta,$ $\alpha=\pi,\epsilon.$
Moreover, since $P\in\mathcal{B}_{\bar{P}_1,\ldots,\bar{P}_l}(d_1,\dots,d_l)$, we have that
\begin{equation}
\label{eq:maininclusion}
P=\text{conv}(P\cup \bar{P}_1\otimes_\pi\cdots\otimes_\pi\bar{P}_l)\cap\bar{P}_1\otimes_\epsilon\cdots\otimes_\epsilon\bar{P}_l,
\end{equation}
for each $P\in\mathcal{A}\subseteq\mathcal{L}_\otimes.$ Thus,
\begin{align*}
\delta^H&\left(\text{conv}(P\cup \bar{P}_1\otimes_\pi\cdots\otimes_\pi\bar{P}_l),\text{conv}(P\cup f^1(\bar{P}_1)\otimes_\pi\cdots\otimes_\pi f^l(\bar{P}_l))\right)\\
&\leq\delta^H\left(\bar{P}_1\otimes_\pi\cdots\otimes_\pi \bar{P}_l,f^1(\bar{P}_1)\otimes_\pi\cdots\otimes_\pi f^l(\bar{P}_l)\right)<\eta.
\end{align*}
Therefore, from (\ref{eq:auxmap}), the identity in (\ref{eq:maininclusion}) and the last inequality, it follows that the map $\tilde{f}_\varepsilon$ associated to the tuple $(f^1,\ldots,f^l)$ satisfies that $\delta^H(\tilde{f}_\varepsilon(P),P)<\varepsilon.$
We claim that $\tilde{f}_\varepsilon(P)\in\mathcal{A},$ for each $P\in\mathcal{A}.$ Indeed, if $P\in\mathcal{A},$
notice that $f^i(\bar{P}_i)\in\mathcal{A}_i,$ $i=1,\dots,l,$ and $\tilde{f}_\varepsilon(P)\in\mathcal{B}_{f^1(\bar{P}_1),\ldots,f^l(\bar{P}_l)}(d_1,\dots,d_l)$. Then, by \cite[(4.1)]{topologytensorialbodies}, 
$conv_\otimes(\tilde{f}_\varepsilon(P))=f^1(\bar{P}_1)\otimes_\pi\cdots\otimes_\pi f^l(\bar{P}_l)$ belongs to $\otimes_\pi(\mathcal{A}_1,\ldots,\mathcal{A}_l).$ This shows that $\tilde{f}_\varepsilon$ is a map $\varepsilon$-close  to the identity of $\mathcal{A}.$

It remains to prove that $\tilde{f}_\varepsilon:\mathcal{A}\rightarrow\mathcal{A}$ is continuous. Indeed, this follows directly from the continuity of the tensor products $\otimes_\pi$, $\otimes_\epsilon$ on $\mathcal{B}(d_1)\times\cdots\times\mathcal{B}(d_1)$ \cite[Proposition 3.3]{topologytensorialbodies}, and  the continuity of the operations convex hull, union and intersection on the hyperspace of $0$-symmetric convex bodies.
\end{proof}

In \cite[Lemma 5.6]{antonyananerviowest}, it is proved that $\{B_2^d\},$ $d\geq2,$ is a $Z$-set in $L(d).$ To achieve this, a family of maps $\mathcal{X}_\rho:L(d)\rightarrow L_0(d),$ $\rho>0,$ $\rho$-close to the identity of $L(d)$ is constructed. Here $L_0(d):=L(d)\setminus\{B_2^d\}.$ 
In the next proposition, we use Lemma \ref{lem:metalema} to prove that
$$
\mathcal{L}_\otimes^{*}(d_1,\ldots,d_l):=\left\{P\in\mathcal{L}_\otimes(d_1,\ldots,d_l): \bar{P}_i\in L_0(d_i), i=1,\ldots,l\right\}
$$
is the complement of a $Z$-set in $\mathcal{L}_\otimes(d_1,\ldots,d_l).$

\begin{proposition}
\label{prop:zetaset}
For any $\varepsilon>0,$ there exists a continuous map $\tilde{\mathcal{X}}_\varepsilon:\mathcal{L}_\otimes(d_1,\ldots,d_l)\\\rightarrow\mathcal{L}_\otimes^{*}(d_1,\ldots,d_l)$ $\varepsilon$-close to the identity of $\mathcal{L}_\otimes(d_1,\ldots,d_l).$ In fact, $\mathcal{L}_\otimes(d_1,\ldots,d_l)\setminus\mathcal{L}_\otimes^{*}(d_1,\ldots,d_l)$ is a $Z$-set in $\mathcal{L}_\otimes(d_1,\ldots,d_l).$
\end{proposition}

\begin{proof}
To shorten notation we  simply write $\mathcal{L}_\otimes^{*}.$ 
First, notice that $\mathcal{L}_\otimes^{*}$  satisfies the following:
\begin{equation}
\label{eq:caractL-0}
\mathcal{L}_\otimes^{*}=(\text{conv}_{\otimes|\mathcal{L}_{\otimes}})^{-1}(\otimes_\pi\left(L_0(d_1),\ldots L_0(d_l))\right).
\end{equation}
This is straightforward from the definition of the $l$-tuple $\bar{P}_i,$ $i=1,\ldots,l,$ associated to each $P\in\mathcal{L}_\otimes$. Furthermore,  $\mathcal{L}_\otimes^{*}$ is open due to the fact that $\otimes_{\pi|L(d_1)\times\cdots\times L(d_l)}$ is a homeomorphism (Proposition \ref{prop:Ltensorpoyectivecontraible}), $\text{conv}_{\otimes|\mathcal{L}_{\otimes}}$ is continuous and each $L_0(d_i)$ is open in $L(d_i)$. As a consequence, its complement $\mathcal{L}_\otimes\setminus\mathcal{L}_\otimes^*$ is closed in $\mathcal{L}_\otimes.$

Now, if we consider the family of continuous maps $\rho$-close to the identity of $L(d_i)$ constructed in \cite[Lemma 5.6]{antonyananerviowest}, $\mathcal{X}^i_\rho:L(d_i)\rightarrow L_0(d_i),$ $\rho>0$ and $i=1,\ldots,l$. Then, from Lemma \ref{lem:metalema}, we know that, for every $\varepsilon>0,$ there exists $\delta>0$ such that the map $\tilde{\mathcal{X}}_\varepsilon:\mathcal{L}_\otimes\rightarrow\mathcal{L}_\otimes,$ defined as
$$
\tilde{\mathcal{X}}_\varepsilon(P)=\text{conv}(P\cup \mathcal{X}_\delta^1(\bar{P}_1)\otimes_\pi\cdots\otimes_\pi \mathcal{X}_\delta^l(\bar{P}_l))\cap \mathcal{X}_\delta^1(\bar{P}_1)\otimes_\epsilon\cdots\otimes_\epsilon \mathcal{X}_\delta^l(\bar{P}_l),
$$
is continuous and $\varepsilon$-close to the identity of $\mathcal{L}_\otimes.$ We have used that
\begin{equation*}
\mathcal{L}_\otimes=(\text{conv}_{\otimes|\mathcal{L}_{\otimes}})^{-1}(\otimes_\pi\left(L(d_1),\ldots L(d_l))\right).
\end{equation*}
The proof is completed by showing that $\tilde{\mathcal{X}}_\varepsilon(P)\in\mathcal{L}_\otimes^*,$ for any $P\in\mathcal{L}_\otimes.$ This follows from the fact that $\tilde{\mathcal{X}}_\varepsilon(P)$ is a tensorial body with respect to $\mathcal{X}_\delta^1(\bar{P}_1),\ldots,\mathcal{X}_\delta^l(\bar{P}_l),$ and so $conv_\otimes(\tilde{\mathcal{X}}_\varepsilon(P))=\mathcal{X}_\delta^1(\bar{P}_1)\otimes_\pi\cdots\otimes_\pi \mathcal{X}_\delta^l(\bar{P}_l)$ (see \cite[(4.1)]{topologytensorialbodies}). 
\end{proof}
In the proof of the theorem below, by $\mathcal{P}(d),$ $d\geq2,$ we denote the subset of $L_0(d)$ consisting of the $P\in L_0(d)$ such that the contact set $\partial P\cap\partial B_2^d$ has empty interior relative to $\partial B_2^d.$ $\mathcal{P}(d)$ was introduced in \cite{antonyananerviowest}.
\begin{theorem}
\label{thm:maintheorem}
Let $d_i\geq2,$ $i=1,\dots,l$. Then $\mathcal{L}_\otimes(d_1,\dots,d_l)$ is homeomorphic to the Hilbert cube.
\end{theorem}

\begin{proof}
Recall that $\mathcal{L}_\otimes(d_1,\dots,d_l)$ is a compact contractible AR (\cite[Proposition 4.6]{topologytensorialbodies} and Theorem \ref{thm:ANRproperty}), then, by the characterization of $\mathcal{Q}$ as the only $\mathcal{Q}$-manifold which is both compact and contractible  \cite[Theorem 7.5.8]{vanMillbook}, we only need to show that $\mathcal{L}_\otimes(d_1,\dots,d_l)$ is a $\mathcal{Q}$-manifold. For this purpose,  we will prove that $\mathcal{L}_\otimes^{*}(d_1,\ldots,d_l)$ is an open $\mathcal{Q}$-manifold in $\mathcal{L}_\otimes(d_1,\dots,d_l),$ whose complement is a $Z$-set (see Proposition \ref{prop:zetaset}). Then, by  \cite[\textsection 3]{Torunczyk1980}, we will have that $\mathcal{L}_\otimes(d_1,\dots,d_l)$ is a $\mathcal{Q}$-manifold.

We begin by proving that $\mathcal{L}_\otimes^{*}$ is a locally compact ANR. To this end, notice that in the proof of Proposition \ref{prop:zetaset}, we  showed that $\mathcal{L}_\otimes^{*}$ is open in $\mathcal{L}_\otimes.$ Hence, given that $\mathcal{L}_\otimes$ is a compact AR (Theorem \ref{thm:ANRproperty}), $\mathcal{L}_\otimes^{*}$ must be a locally compact ANR.

To prove that $\mathcal{L}_\otimes^{*}$ is a Hilbert cube manifold, we use Toru\'{n}czyk's characterization \cite[Theorem 1]{Torunczyk1980}. In this sense, it is enough to show that for each $\varepsilon>0,$ there are two continuous maps $\tilde{f}_\varepsilon,\tilde{h}_\varepsilon:\mathcal{L}_\otimes^{*}\rightarrow\mathcal{L}_\otimes^{*}$ $\varepsilon$-close  to the identity of $\mathcal{L}_\otimes^{*}$ such that $\text{Im}(\tilde{f}_\varepsilon)\cap\text{Im}(\tilde{h}_\varepsilon)=\emptyset.$
Let $f_\rho^i:L_0(d_i)\rightarrow\mathcal{P}(d_i)$ and $h_\rho^i:L_0(d_i)\rightarrow L_0(d_i)\setminus\mathcal{P}(d_i),$ $\rho>0$, be the families of continuous maps, defined in \cite[Lemma 5.4]{antonyananerviowest}, $\rho$-close to the identity of $L_0(d_i),$ $i=1,\dots,l.$
Observe that, as a consequence of (\ref{eq:caractL-0}), if we let $\mathcal{A}_i=L_0(d_i)$ in Lemma \ref{lem:metalema}, then for every $\varepsilon>0,$ there exists $\delta>0$ such that the maps $\tilde{f}_\varepsilon,\tilde{h}_\varepsilon:\mathcal{L}_\otimes^{*}\rightarrow\mathcal{L}_\otimes^{*}$ defined as
$$
\tilde{f}_\varepsilon(P)=\text{conv}(P\cup f_\delta^1(\bar{P}_1)\otimes_\pi\cdots\otimes_\pi f_\delta^l(\bar{P}_l))\cap f_\delta^1(\bar{P}_1)\otimes_\epsilon\cdots\otimes_\epsilon f_\delta^l(\bar{P}_l),
$$ 
and
$$
\tilde{h}_\varepsilon(P)=\text{conv}(P\cup h_\delta^1(\bar{P}_1)\otimes_\pi\cdots\otimes_\pi h_\delta^l(\bar{P}_l))\cap h_\delta^1(\bar{P}_1)\otimes_\epsilon\cdots\otimes_\epsilon h_\delta^l(\bar{P}_l)
$$
are continuous and $\varepsilon$-close to the identity of $\mathcal{L}_\otimes^{*}.$ Moreover, by \cite[(4.1)]{topologytensorialbodies}, due to $\tilde{f}_\varepsilon(P)$  and $\tilde{h}_\varepsilon(P)$ are tensorial bodies with respect to $f_\delta^1(\bar{P}_1),\dots, f_\delta^l(\bar{P}_l)$ and $h_\delta^1(\bar{P}_1),\dots, h_\delta^l(\bar{P}_l),$ respectively, we have that 
$$
conv_\otimes(\tilde{f}_\varepsilon(P))=f_\delta^1(\bar{P}_1)\otimes_\pi\cdots\otimes_\pi f_\delta^l(\bar{P}_l)\in\otimes_\pi(\text{Im}(f_\delta^1)\times\cdots\times\text{Im}(f_\delta^l)),
$$
and 
$conv_\otimes(\tilde{h}_\varepsilon(P))=h_\delta^1(\bar{P}_1)\otimes_\pi\cdots\otimes_\pi h_\delta^l(\bar{P}_l)\in\otimes_\pi(\text{Im}(h_\delta^1)\times\cdots\times\text{Im}(h_\delta^l)),
$
for any $P\in\mathcal{L}_\otimes^{*}.$ Hence, since $\otimes_{\pi|L(d_1)\times\cdots\times L(d_l)}$ is a homeomorphism (Proposition \ref{prop:Ltensorpoyectivecontraible}),  $\text{Im}(f_\delta^i)\in\mathcal{P}(d_i)$ and  $\text{Im}(h_\delta^i)\in L_0(d_i)\setminus\mathcal{P}(d_i),$ $i=1,\dots,l,$ then the images $\tilde{f}_\varepsilon$ and $\tilde{h}_\varepsilon$ must be disjoint.
 \end{proof}

\begin{corollary}
\label{cor:mainresultcor}
Let $d_i\geq2,$ $i=1,\ldots,l$ and  $p=\frac{d_1(d_1+1)+\cdots+d_l(d_l+1)}{2}.$ Then, the space of tensorial bodies  $\mathcal{B}_\otimes(d_1,\ldots,d_l)$ is homeomorphic to $\mathcal{Q}\times\mathbb{R}^p.$ Particularly, $\mathcal{B}_\otimes(d_1,\ldots,d_l)$ and $\Pi_\otimes(d_1,\ldots,d_l)$ are homeomorphic.
\end{corollary}
\begin{proof}
It is a direct consequence of Theorem \ref{thm:maintheorem} and \cite[Corollary 4.11]{topologytensorialbodies}. In the latter, it is proved that $\mathcal{B}_\otimes(d_1,\ldots,d_l)$ is homeomorphic to $\mathcal{L}_\otimes(d_1,\ldots,d_l)\times\mathbb{R}^p$, with $p$ as in the statement of the corollary. The last part of the statement follows directly from Theorem \ref{thm:structurePi}-3.
\end{proof}

\section{Final remarks}
\label{sec:discussion}

Let us finish the paper by exploring the relation between the so called Banach-Mazur compactum  and the orbit space $\mathcal{B}_\otimes(d_1,\dots,d_l)/GL_\otimes(d_1,\dots,d_l)$. To this end, we first introduce some notation. The Banach-Mazur compactum is denoted by $\mathcal{BM}(d).$ It is the set of equivalence classes of $0$-symmetric convex bodies in $\mathbb{R}^{d}$ determined by the relation: $P\sim R$ if and only if $\delta^{BM}(P,R)=1.$ Here, $\delta^{BM}(\cdot,\cdot)$ is the so called Banach-Mazur distance \cite[p. 309]{Gruber1993a}.
It is well-known that $\mathcal{BM}(d)$ is a compact metric space with the metric induced by the logarithm $\text{log }\delta_{BM}.$ In terms of topological groups, another well-known fact is that the Banach-Mazur compactum $(\mathcal{BM}(d),\text{log }\delta^{BM})$ is homeomorphic to the orbit space $\mathcal{B}(d)/GL(d)$ (see \textit{e.g.} \cite[p. 1191]{LindenstraussMilman1993}) or \cite[Section 6]{Schaffer1967} and the references therein). 
In a similar way, for the class of tensorial bodies a Banach-Mazur type compactum can be defined. In this case, the compactum of tensorial bodies $\mathcal{BM}_\otimes(d_1,\dots,d_l)$ is defined as the set of equivalence classes of tensorial bodies in $\otimes_{i=1}^{l}\mathbb{R}^{d_i}$ determined by the relation: $P\sim' R$ if and only if  $\delta_{\otimes}^{BM}(P,R)=1.$ Here, $\delta_{\otimes}^{BM}(\cdot,\cdot)$ denotes the tensorial Banach-Mazur distance (\ref{eq:tensorial bm distance}).
In \cite{topologytensorialbodies}, it is proved that $\text{log }\delta_\otimes^{BM}$ is a metric on $\mathcal{BM}_\otimes(d_1,\dots,d_l)$ which turns it into a compact metric space. Indeed, a fundamental property is that $(\mathcal{BM}_\otimes(d_1,\dots,d_l),\delta_{\otimes}^{BM})$ is homeomorphic to the orbit space $\mathcal{B}_\otimes(d_1,\dots,d_l)/GL_\otimes(d_1,\dots,d_l)$ (see \cite[Corollary 4.9]{topologytensorialbodies}). For this reason, we also refer to this orbit space as the compactum of tensorial bodies.

In \cite{banacamazurcompactum,antonyananerviowest,SergeyCorrigendum} it is proved that $\mathcal{B}(d)$ is a $GL(d)$-proper space for which $L(d)$ is a compact $O(d)$-global slice.  This yields to new topological models for $\mathcal{BM}(d),$ which turns out to be homeomorphic to the orbit space $L(d)/O(d).$ Since $L(d)$ and $\mathcal{Q}$ are homeomorphic, a representation of $\mathcal{BM}(d)$ as an orbit space $\mathcal{Q}/O(d)$ is obtained. In the same spirit, in the context of tensorial bodies, we know that $\mathcal{B}_\otimes(d_1,\dots,d_l)$ is a $GL_\otimes$-proper space for which $\mathcal{L}_\otimes(d_1,\dots,d_l)$ is a compact $O_\otimes$-global slice, and the corresponding orbit spaces,  $\mathcal{L}_\otimes/O_\otimes$ and $\mathcal{B}_\otimes/GL_\otimes,$ are homeomorphic \cite{topologytensorialbodies}. Moreover, by \cite[Corollary 4.9]{topologytensorialbodies} and Theorem \ref{thm:maintheorem}, the compactum $\mathcal{BM}_\otimes(d_1,\dots,d_l)$ can be represented as an orbit space $\mathcal{Q}/O_\otimes(d_1,\dots,d_l).$ 

In the following proposition we show that, as occurs with the Banach-Mazur compactum, the compactum of tensorial bodies is contractible.
\begin{proposition}
\label{prop_BMtensorcontractible}
Let $d_i\geq2,$ $i=1,\dots,l,$ $\mathcal{BM}_\otimes(d_1,\dots,d_l)$ is contractible.
\end{proposition}
\begin{proof}
We will define an $O_\otimes$-homotopy from $\mathcal{L}_\otimes(d_1,\dots,d_l)$ to the $O_\otimes$-fixed point $B_2^{d_1}\otimes_\pi\cdots\otimes_\pi B_2^{d_l},$ as a consequence the orbit space $\mathcal{L}_\otimes/O_\otimes$ will be  contractible to the point $\{B_2^{d_1}\otimes_\pi\cdots\otimes_\pi B_2^{d_l}\}.$  The result then will follow from the fact that $\mathcal{BM}_\otimes$ is homeomorphic $\mathcal{L}_\otimes/O_\otimes$ (\cite[Corollary 4.9]{topologytensorialbodies}). 
Let $H:\mathcal{L}_{\otimes}\times[0,1]\rightarrow\mathcal{L}_{\otimes}$ be the $GL_\otimes$-homotopy used in the proof of Proposition \ref{prop:Ltensorpoyectivecontraible}-1, and let $F:\Pi\cap\mathcal{L}_{\otimes}\times[0,1]\rightarrow\Pi\cap\mathcal{L}_{\otimes}$  be the $O_\otimes$-homotopy used in the proof of  Proposition \ref{prop:Ltensorpoyectivecontraible}-3. Then, the map $G:\mathcal{L}_\otimes\times[0,1]\rightarrow\mathcal{L}_\otimes$ defined as
 \begin{equation*}
 G(P,t)=\begin{cases}H(P,2t); \text{ if }t\in\left[0,\frac{1}{2}\right],\\
 F(conv_\otimes(P),2t-1);\text{ if }t\in\left[\frac{1}{2},1\right].
 \end{cases}
 \end{equation*}
is the desired homotopy. Indeed, since $H(P,1)=conv_\otimes(P)=F(conv_\otimes(P),0),$ $G$ is well-defined. Its continuity and $O_\otimes$-equivariance follow directly from the properties of $H$ and $F.$   Clearly, $G(P,0)=P$ and $G(P,1)=B_2^{d_1}\otimes_\pi\cdots\otimes_\pi B_2^{d_l},$ for every 
$Q\in\mathcal{L}_{\otimes}.$
 \end{proof}

In \cite[Theorem 3.3]{SergeyBanachMazurmodel}, some conditions on the action of the orthogonal group $O(d)$ on a Hilbert cube $X$ guaranteeing that the orbit space $X/O(d)$ is homeomorphic to $\mathcal{BM}(d)$ are established. There an essential hypothesis is that the action must have a unique $O(d)$-fixed point.  In the case of the compactum of tensorial bodies, represented as $\mathcal{L}_\otimes(d_1,\dots,d_l)/O_\otimes(d_1,\dots,d_l),$ the group $O_\otimes(d_1,\dots,d_l)$ is a maximal compact subgroup of  $GL_\otimes(d_1,\dots,d_l)$ consisting of orthogonal maps \cite{topologytensorialbodies}, for which the $O_\otimes$-space $\mathcal{L}_\otimes(d_1,\dots,d_l)$ has at least three $O_\otimes$-fixed points. 

\begin{remark}
\label{rem:notunique}
The Euclidean ball $B_2^{d_1,\dots,d_l},$ and the tensor products $B_2^{d_1}\otimes_\pi\cdots\otimes_\pi B_2^{d_l}$ and $B_2^{d_1}\otimes_\epsilon\cdots\otimes_\epsilon B_2^{d_l}$ are $O_\otimes$-fixed points in $\mathcal{L}_\otimes(d_1,\dots,d_l).$ 
\end{remark}
\begin{proof}
First notice that  $B_2^{d_1,\dots,d_l}, B_2^{d_1}\otimes_\pi\cdots\otimes_\pi B_2^{d_l}$ and $B_2^{d_1}\otimes_\epsilon\cdots\otimes_\epsilon B_2^{d_l}$ are tensorial bodies w.r.t. $B_2^{d_1},\dots,B_2^{d_l},$ then, by \cite[Proposition 4.4 (2)]{topologytensorialbodies},
$$
\ell_\otimes(B_2^{d_1,\dots,d_l})=\ell_\otimes(B_2^{d_1}\otimes_\epsilon\cdots\otimes_\epsilon B_2^{d_l})
=\ell_\otimes(B_2^{d_1}\otimes_\pi\cdots\otimes_\pi B_2^{d_l})=B_2^{d_1,\dots,d_l}.
$$
Hence, all of them belong to $\mathcal{L}_\otimes(d_1,\dots,d_l).$
Since $U\in O_\otimes(d_1,\dots,d_l)$ is an orthogonal map  $U(B_2^{d_1,\dots,d_l})=B_2^{d_1,\dots,d_l},$ and thus $B_2^{d_1,\dots,d_l}$ is an $O_\otimes$-fixed point.
 On the other hand, since any $U\in O_\otimes(d_1,\dots,d_l)$ can be written as $U=(U_{1}\otimes\cdots\otimes U_{l})U_\sigma,$ for some $U_i\in O(d_i),$ $i=1,\dots,l,$ and $U_\sigma\in\mathcal{P}$ \cite[Proposition A.4]{topologytensorialbodies}, then
\begin{align*}
U(B_2^{d_1}\otimes_\alpha\cdots\otimes_\alpha B_2^{d_l})=U(B_2^{d_1}\otimes_{\alpha}\cdots\otimes_{\alpha}B_2^{d_l})&= U_{1}(B_2^{d_{\sigma(1)}})\otimes_{\alpha}\cdots\otimes_{\alpha}U_{l}(B_2^{d_{\sigma(l)}})\\&= B_2^{d_{\sigma(1)}}\otimes_{\alpha}\cdots\otimes_{\alpha}B_2^{d_{\sigma(l)}}
\end{align*}
for $\alpha=\pi,\epsilon,$ see equality (\ref{eq:gltensor image of proj and inj tp}). Now, given that the permutation on the dimensions $d_{i}$'s is only possible if some of them  are the same, we must have that $B_2^{d_{\sigma(1)}}\otimes_{\alpha}\cdots\otimes_{\alpha}B_2^{d_{\sigma(l)}}=B_2^{d_1}\otimes_\alpha\cdots\otimes_\alpha B_2^{d_l}$ and $U(B_2^{d_1}\otimes_\alpha\cdots\otimes_\alpha B_2^{d_l})=B_2^{d_1}\otimes_\alpha\cdots\otimes_\alpha B_2^{d_l}$ for $\alpha=\pi,\epsilon$ and any $U\in O_\otimes,$ as desired. 
\end{proof}

Taking into account that the $GL_\otimes$-spaces  $\mathcal{B}_\otimes(d_1,\dots,d_l)$ and $\Pi(d_1,\dots,d_l)$ are homeomorphic $\mathcal{Q}$-manifolds (Corollary \ref{cor:mainresultcor}),  the following natural question arises:

\

\noindent\textbf{Question 1} Are $\mathcal{B}_\otimes(d_1,\dots,d_l)/GL_\otimes(d_1,\dots,d_l)$ and $\Pi(d_1,\dots,d_l)/GL_\otimes(d_1,\dots,d_l)$ homeomorphic spaces? In particular: Is the compactum $\mathcal{BM}_\otimes(d_1,\dots,d_l)$ homeomorphic to $\Pi(d_1,\dots,d_l)/GL_\otimes(d_1,\dots,d_l)$?

\

Bringing together Proposition \ref{prop:convisproper} and Corollary \ref{cor:mainresultcor}, we have that the  $GL_\otimes$-map $conv_\otimes:\mathcal{B}_\otimes(d_1,\dots,d_l)\rightarrow\Pi(d_1,\dots,d_l)$ is indeed a cell-like map between $\mathcal{Q}$-manifolds. In consequence, it is also a near-homeomorphism \cite[Corollary 43.2]{Chapmanbook}. This motivates the following question:

\

\noindent\textbf{Question 2} Does there exist a $GL_\otimes$-homeomorphism between  $\mathcal{B}_\otimes(d_1,\dots,d_l)$ and $\Pi(d_1,\dots,d_l)$?

\

Clearly, a positive answer to the last question will provide a positive answer to the first one. We finish the paper by exhibiting the relation between the Banach-Mazur compactum and the orbit space $\Pi(d_1,\dots,d_l)/GL_\otimes(d_1,\dots,d_l)$ when the dimensions $d_i$'s are different from each other.

\begin{proposition}
\label{prop:orbitspacePI}
 The orbit  space $\Pi(d_1,\dots,d_l)/GL_\otimes(d_1,\dots,d_l)$ is homeomorphic to  the product $\mathcal{BM}(d_1)\times\cdots\times\mathcal{BM}(d_l),$ as long as $d_i\neq d_j$ for $i\neq j,$ $i,j=1,\dots,l.$ 
\end{proposition}
\begin{proof}
Consider the tensor product $\otimes_\pi:\mathcal{B}(d_1)\times\cdots\times\mathcal{B}(d_l)\rightarrow\Pi(d_1,\dots,d_l).$ We already know that $\otimes_\pi$ is a continuous surjective map \cite[Proposition 3.3]{tensorialbodies}. Moreover, for every $(P_1,\dots,P_l)$ and $(R_1,\dots,R_l)$ in $\mathcal{B}(d_1)\times\cdots\times\mathcal{B}(d_l),$  if $P_i$ belongs to the $GL(d_i)$-orbit of $R_i,$ $i=1,\ldots,l,$ then $P_i=T_iR_i$ for some $T_i\in GL(d_i),$ $i=1,\dots,l,$ and 
$$
P_1\otimes_\pi\cdots\otimes_\pi P_l=T_1\otimes\cdots\otimes T_l(R_1\otimes_\pi\cdots\otimes_\pi R_l).
$$
See equation (\ref{eq:gltensor image of proj and inj tp}). Hence, $P_1\otimes_\pi\cdots\otimes_\pi P_l$ belongs to the $GL_\otimes$-orbit of $R_1\otimes_\pi\cdots\otimes_\pi R_l.$ On the other side, 
since $d_i\neq d_j,$ for $i\neq j,$ then any map $T\in GL_\otimes(d_1,\dots,d_l)$ can be written as $T=T_1\otimes\cdots\otimes T_l$ for some $T_i\in GL(d_i),$ $i=1,\dots,l$ (see Section \ref{sec:Gspaces}). This, together with equation (\ref{eq:gltensor image of proj and inj tp}), shows that if $P_1\otimes_\pi\cdots\otimes_\pi P_l$ belongs to the $GL_\otimes$-orbit of $R_1\otimes_\pi\cdots\otimes_\pi R_l,$ then there exists $T_i\in GL(d_i),$ $i=1,\dots,l$ such that $T_1R_1\otimes_\pi\cdots\otimes_\pi T_lR_l=P_1\otimes_\pi\cdots\otimes_\pi P_l.$ Thus, from \cite[Proposition 3.6]{tensorialbodies}, there exists $\lambda_i>0,$ $i=1,\dots,l,$ such that $\lambda_1\cdots\lambda_l=1$ and $T_iR_i=\lambda_i P_i$ for every $i.$ Thus, by the identity (\ref{eq:projectivelambda}),
$$
P_1\otimes_\pi\cdots\otimes_\pi P_l=\lambda_1P_1\otimes_\pi\cdots\otimes_\pi\lambda_lP_l=T_1\otimes\cdots\otimes T_l(R_1\otimes_\pi\cdots\otimes_\pi R_l).
$$
Therefore, $P_1\otimes_\pi\cdots\otimes_\pi P_l$ belongs to the $GL_\otimes$-orbit of $R_1\otimes_\pi\cdots\otimes_\pi R_l$ if and only if each $P_i$ belongs to the $GL(d_i)$-orbit of $R_i.$ In consequence, the map $\otimes_\pi$ induces a homeomorphism between the orbit space $\Pi(d_1,\dots,d_l)/GL_\otimes(d_1,\dots,d_l)$  and $\mathcal{BM}(d_1)\times\cdots\times\mathcal{BM}(d_l).$
 \end{proof}





\begin{thebibliography}{10}
\bibliographystyle{acm}

\bibitem{banacamazurcompactum}
{\sc {A}ntonyan, S.}
\newblock {T}he topology of the {B}anach-{M}azur compactum.
\newblock {\em {F}und. {M}ath. 166}, 3 (2000), 209--232.

\bibitem{antonyananerviowest}
{\sc {A}ntonyan, S.}
\newblock {W}est's problem on equivariant hyperspaces and {B}anach-{M}azur
  compacta.
\newblock {\em {T}rans. {A}mer. {M}ath. {S}oc. 355}, 8 (2003), 3379--3404.

\bibitem{SergeyBanachMazurmodel}
{\sc {A}ntonyan, S.}
\newblock New topological models for {B}anach-{M}azur compacta.
\newblock {\em Fundam. Prikl. Mat. 11}, 5 (2005), 19--31.

\bibitem{SergeyCorrigendum}
{\sc {A}ntonyan, S.}
\newblock Corrigendum to: ``{W}est's problem on equivariant hyperspaces and
  {B}anach-{M}azur compacta'' [{T}rans. {A}mer. {M}ath. {S}oc. {355} (2003),
  no. 8, 3379--3404; mr1974693].
\newblock {\em Trans. Amer. Math. Soc. 358}, 12 (2006), 5631--5633.

\bibitem{AntonyanNatalia}
{\sc {A}ntonyan, S., and {J}onard{-}P{\'e}rez, N.}
\newblock Affine group acting on hyperspaces of compact convex subsets of {$\mathbb{R}^n$}.
\newblock {\em {F}und. {M}ath. 223\/} (2013), 99--136.

\bibitem{antonyanjonardordonez}
{\sc Antonyan, S., {J}onard{-}P{\'e}rez, N., and Ju\'{a}rez-Ord\'{o}\~{n}ez,
  S.}
\newblock Hyperspaces of convex bodies of constant width.
\newblock 347--361.

\bibitem{xorgames}
{\sc {A}ubrun, G., {L}ami, L., {P}alazuelos, C., {S}zarek, S., and {W}inter,
  A.}
\newblock Universal gaps for {XOR} games from estimates on tensor norm ratios.
\newblock {\em {C}omm. {M}ath. {P}hys. 375}, 1 (2020), 679--724.

\bibitem{Aubrun2006}
{\sc {A}ubrun, G., and Szarek, S.}
\newblock Tensor products of convex sets and the volume of separable states on
  n qudits.
\newblock {\em Physical Review A 73}, 2 (2006), 022109.

\bibitem{Aubrun2017}
{\sc {A}ubrun, G., and {S}zarek, S.}
\newblock {\em {A}lice and {B}ob {M}eet {B}anach: {T}he {I}nterface of
  {A}symptotic {G}eometric {A}nalysis and {Q}uantum {I}nformation {T}heory},
  vol.~223.
\newblock {A}merican {M}athematical {S}oc., 2017.

\bibitem{Bazilevich97}
{\sc Bazilevich, L.}
\newblock Topology of a hyperspace of convex bodies of constant width.
\newblock {\em Math. {N}otes 62}, 6 (1997), 683--697.

\bibitem{Bazylevych2006}
{\sc Bazylevych, L., and Zarichnyi, M.}
\newblock On convex bodies of constant width.
\newblock {\em Topology {A}ppl. 153}, 11 (2006), 1699--1704.

\bibitem{Belegradek2018}
{\sc Belegradek, I.}
\newblock Hyperspaces of smooth convex bodies up to congruence.
\newblock {\em Adv. Math. 332\/} (2018), 176--198.

\bibitem{BessagaPelczynski}
{\sc {B}essaga, {C.} and {P}e\l czy\'{n}ski, {A}.}
\newblock {\em {S}elected {T}opics in {I}nfinite-{D}imensional                {T}opology}.
\newblock PWN---Polish Scientific Publishers, Warsaw, 1975.

\bibitem{Bredon}
{\sc {B}redon, G.}
\newblock {\em {I}ntroduction to {C}ompact {T}ransformation {G}roups}.
\newblock {A}cademic {P}ress, {N}ew {Y}ork-{L}ondon, 1972.

\bibitem{Chapmanbook}
{\sc Chapman, T.}
\newblock {\em Lectures on {H}ilbert {C}ube {M}anifolds}.
\newblock {A}merican {M}athematical {S}ociety, {P}rovidence, {R. I.}, 1976.
\newblock Expository lectures from the {CBMS} {R}egional {C}onference held at
  {G}uilford {C}ollege, {O}ctober 11-15, 1975, {R}egional {C}onference {S}eries
  in {M}athematics, No. 28.

\bibitem{defantfloret}
{\sc Defant, K., and Floret, K.}
\newblock {\em Tensor {N}orms and {O}perator {I}deals}.
\newblock {N}orth {H}olland {M}athematics {S}tudies, 1992.

\bibitem{maite}
{\sc {F}ern\'andez{-U}nzueta, M.}
\newblock The {S}egre cone of {B}anach spaces and multilinear mappings.
\newblock {\em {L}inear {M}ultilinear {A}lgebra\/} (2018).
\newblock 10.1080/03081087.2018.1509938.

\bibitem{tensorialbodies}
{\sc {F}ern{\'a}ndez{-U}nzueta, M., and {H}igueras{-M}onta{\~n}o, L.}
\newblock {C}onvex {B}odies {A}ssociated to {T}ensor {N}orms.
\newblock {\em {J}. {C}onvex {A}nal. 26}, 4 (2019), 1297--1320.

\bibitem{maiteluisa2}
{\sc {F}ern{\'a}ndez{-U}nzueta, M., and {H}igueras{-M}onta{\~n}o, L.}
\newblock A general theory of tensor products of convex sets in {E}uclidean
  spaces.
\newblock {\em Positivity 24}, 5 (2020), 1373--1398.

\bibitem{Gruber1993a}
{\sc Gruber, P.}
\newblock {T}he {S}pace of {C}onvex Bodies. {\em {I}n {H}andbook of {C}onvex {G}eometry {V}ol. {A}, {B}}.
\newblock {N}orth-{H}olland, {A}msterdam, 1993, 301--318.

\bibitem{topologytensorialbodies}
{\sc Higueras-Monta\~{n}o, L.}
\newblock A hyperspace of convex bodies arising from tensor norms.
\newblock {\em {T}opology and {A}ppl. 275\/} (2020), 107149, 20.

\bibitem{kadisonkingrose}
{\sc {K}adison, R., and {R}ingrose, J.}
\newblock {\em {F}undamentals of the {T}heory of {O}perator {A}lgebras}.
\newblock {A}cademic {P}ress, 1983.

\bibitem{Macbeath51}
{\sc Macbeath, A.}
\newblock A compactness theorem for affine equivalence-classes of convex
  regions.
\newblock {\em Canad. J. Math. 3\/} (1951), 54--61.

\bibitem{vanMillbook}
{\sc van {M}ill, J.}
\newblock {\em Infinite-{D}imensional {T}opology: {P}rerequisites and {I}ntroduction.}
\newblock North-Holland Publishing Co., Amsterdam, 1989.

\bibitem{nadler1979}
{\sc Nadler, S., Quinn, J., and Stavrakas, N.}
\newblock Hyperspaces of compact convex sets.
\newblock {\em {P}acific {J}. {M}ath. 83}, 2 (1979), 441--462.

\bibitem{LindenstraussMilman1993}
{\sc Lindenstrauss, {J}. and Milman, {V}. {D}.}
\newblock The local theory of normed spaces and its applications to convexity.
\newblock In {\em {H}andbook of {C}onvex {G}eometry, {V}ol. {A}, {B}}. {N}orth-{H}olland, {A}msterdam, 1993, 1149--1220.

\bibitem{Palais1}
{\sc {P}alais, R.}
\newblock {\em The classification of {$G$}-spaces}.
\newblock Mem. Amer. Math. Soc. No. 36, 1960.

\bibitem{Palais1961}
{\sc Palais, R.}
\newblock On the existence of slices for actions of non-compact {L}ie groups.
\newblock {\em {A}nn. of {M}ath. (2) 73\/} (1961), 295--323.

\bibitem{PearsAR}
{\sc Pears, {A}.{R}.}
\newblock {\em {D}imension {T}heory of {G}eneral {S}paces.}
\newblock {C}ambridge {U}niversity {P}ress, {C}ambridge,                    {E}ngland-{N}ew {Y}ork-{M}elbourne, 1975.

\bibitem{Ryan2013}
{\sc {R}yan, R.}
\newblock {\em {I}ntroduction to {T}ensor {P}roducts of {B}anach {S}paces}.
\newblock {S}pringer {M}onographs in {M}athematics, 2002.

\bibitem{Schaffer1967}
{\sc Sch\"{a}ffer, {J}.}
\newblock {{I}nner diameter, perimeter, and girth of spheres.}
\newblock {\em {M}ath. {A}nn. 173 (1967), 59--79; addendum, ibid.} 173 1967 79--82. 

\bibitem{Schneider1993}
{\sc {S}chneider, R.}
\newblock {\em Convex {B}odies: {T}he {B}runn-{M}inkowski {T}heory.}
\newblock {C}ambridge {U}niversity {P}ress, {C}ambridge, 1993.
  
\bibitem{ThompsonAC}
{\sc {T}hompson, {A.}{C.}}
\newblock {\em {M}inkowski {G}eometry}.
\newblock {C}ambridge {U}niversity {P}ress, {C}ambridge, 1996.

\bibitem{Tomczak-Jaegermann1989}
{\sc Tomczak-Jaegermann, N.}
\newblock {\em {B}anach-{M}azur {D}istances and {F}inite-{D}imensional
  {O}perator {I}deals}.
\newblock Longman Sc \& Tech, 1989.

\bibitem{Torunczyk1980}
{\sc Toru\'{n}czyk, H.}
\newblock On {${\rm CE}$}-images of the {H}ilbert cube and characterization of
  {$Q$}-manifolds.
\newblock {\em Fund. Math. 106}, 1 (1980), 31--40.


\end{thebibliography}
\end{document}